\documentclass[reqno,a4paper]{amsart}
\usepackage[colorlinks,citecolor=blue,urlcolor=blue]{hyperref}
\RequirePackage{amsthm,amsmath,amsfonts,amssymb}
\RequirePackage[numbers]{natbib}
\numberwithin{equation}{section}
\theoremstyle{plain}
\newtheorem{theo}{Theorem}[section]
\newtheorem{prop}[theo]{Proposition}

\newtheorem{cond}[theo]{Condition}
\newtheorem{lemme}[theo]{Lemma}

\newtheorem{cor}[theo]{Corollary}

\theoremstyle{definition}

\newtheorem{remark}[theo]{Remark}

\newtheorem{defi}[theo]{Definition}

\usepackage{yfonts}
\usepackage{enumerate}
\RequirePackage[colorlinks,citecolor=blue,urlcolor=blue]{hyperref}

\usepackage{tikz}
\usetikzlibrary{intersections,shapes,trees,arrows,spy,positioning,decorations,arrows.meta,decorations.pathreplacing,patterns,calc,decorations.markings}
\tikzset
{marking1/.style=
	{decoration=
		{markings,
			mark= between positions 0.03 and 0.97 step 5 mm with {\arrow[line width=0.5pt]{>}}
		},
		postaction=decorate
	}
}
\usepackage{mathrsfs}
\usepackage{color}
\usepackage{mathtools}
\usepackage{euscript}
\usepackage[applemac]{inputenc}
\usepackage{dsfont}
\usepackage{graphics,graphicx}
\usepackage[T1]{fontenc}
\usepackage{color}
\usepackage{epsfig,psfrag}

\def\E{\mathbb{E}}

\def\P{\mathbb{P}}

\def\R{\mathbb{R}}

\newcommand\Bc{{\mathscr B}}
\newcommand\Cs{{\mathcal C}}

\def\dd{\textnormal{d}}
\newcommand{\bindist}[2]{\textrm{Bin}({#1},{#2})}
 
\newcommand{\berno}[1]{\textrm{Bern}({#1})}
\newcommand{\Eb}  {{\mathbb E}}

\newcommand{\Nb}  {{\mathbb N}}

\newcommand{\Pb}  {{\mathbb P}}

\newcommand{\Rb}  {{\mathbb R}}

\newcommand{\Ls} {{\mathcal L}}
\newcommand{\Ms} {{\mathcal M}}

\newcommand*{\affmark}[1][*]{\textsuperscript{#1}}

\title[Ergodic processes with a Bernstein or a Siegmund dual]{Distance to stationarity and open set recurrence for ergodic processes on the unit interval with a Bernstein or a Siegmund dual}

\author{Fernando Cordero\affmark[1] \and Gr\'egoire V\'echambre\affmark[2]}
\address{\newline \affmark[1]BOKU University, Department of Natural Sciences and Sustainable Resources, Institute of Mathematics, Austria
\newline\affmark[2]State Key Laboratory of Mathematical Sciences, Academy of Mathematics and Systems Science, Chinese Academy of Sciences, China}
\email{\affmark[1]fernando.cordero@boku.ac.at, \affmark[2]vechambre@amss.ac.cn}

\date{\today}

\begin{document} 
\maketitle

\begin{abstract}
We use duality techniques - specifically Siegmund and Bernstein duality - as tools to analyse ergodic and recurrence properties of $[0,1]$-valued Markov processes. These dualities enable the derivation of sharp bounds on the distance to stationarity and allow for a novel approach to establishing topological recurrence, whereby ergodicity implies recurrent visits to neighbourhoods of points in the stationary distribution's support.
We first relate Siegmund and Bernstein duality in cases where both apply, such as $\Lambda$-Wright--Fisher processes. We then exploit the dual processes to derive bounds on the distance to stationarity and simple criteria for open set recurrence. Our results apply to a broad class of $\Lambda$-Wright--Fisher processes with frequency-dependent selection, mutation, and random environmental effects.
In many cases, Bernstein duals can be constructed from the same ancestral structures underlying moment duals, allowing our methods to extend and strengthen existing results based on moment duality. This framework provides a robust approach to studying the long-term behaviour of complex population models beyond the classical moment dual setting.
\end{abstract}
\textbf{Keywords:} Bernstein duality; Siegmund duality; distance to stationarity; open set recurrence; $\Lambda$-Wright--Fisher process; frequency-dependent selection; coordination; random environment

\bigskip

\section{Introduction}
The notion of duality for Markov processes has proven to be a powerful tool in a variety of fields, including interacting particle systems \cite{holey}, queueing theory \cite{shan,Green}, mathematical finance \cite{Lo17}, stochastic differential equations \cite{HSW77,willi78}, and population genetics \cite{EGT10,BW18,BCH18,casa19e}. Duality relates the one-dimensional distributions of two Markov processes, $X \coloneqq (X_t)_{t\geq 0}$ and $Y \coloneqq (Y_t)_{t\geq 0}$. For simplicity, we restrict our discussion to continuous-time processes with state spaces $E$ and $F$, respectively. In this setting, duality asserts that
\[
\Eb[H(X_t, y)\mid X_0 = x] = \Eb[H(x, Y_t)\mid Y_0 = y],\qquad t \geq 0,\ x \in E,\ y \in F,
\]
where $H \colon E \times F \to \Rb$ is a measurable function called the \emph{duality function}. Typically, the focus lies on one of the processes, say $X$, which is often more difficult to analyse. In contrast, the dual process $Y$ is typically more tractable and is used to infer properties of $X$. For instance, if $H$ contains sufficient information about the processes, one can deduce long-term behaviour of $X$ from that of $Y$. Duality has also been successfully used to establish existence and uniqueness of solutions to martingale problems.

This work is motivated by applications in population genetics, where the process $X$ describes the evolution of the proportion of individuals of a given type in a two-type population. Accordingly, we assume that $X$ takes values in the interval $[0,1]$. The state space of $Y$ depends heavily on the type of duality under consideration. For instance, in the case of a \emph{moment duality}, where the duality function is of the form $H(x, n) = x^n$, the state space of $Y$ is typically $\Nb$ (or $\Nb_0$, or $\Nb \cup \{\Delta\}$, with $\Delta$ a cemetery state). In population genetics, such a process $Y$ is often interpreted as an ancestral process. Moment dualities are valuable tools for analysing $[0,1]$-valued processes because positive integer moments uniquely characterize distributions on $[0,1]$ (see e.g. \cite{foucart2013impact,CF21,GS18,CorderoVechambre,casa19e,BEH21}). However, the existence of a moment dual imposes strong constraints on the process $X$ and is therefore more the exception than the rule.

To overcome this limitation, a new type of duality, referred to as \emph{Bernstein duality}, was introduced in \cite{Cordero2022} (see also \cite{Vec2023} for another generalization of moment duality). Bernstein duality was used to analyse $\Lambda$-Wright--Fisher processes with polynomial frequency-dependent selection, which lie beyond the scope of moment duality. This framework was further extended in \cite{cordhumvech2024} to include models with mutation and environmental randomness with opposing effects (see also \cite{Koske24} for broader forms of frequency-dependent selection). Moreover, Bernstein duality appears particularly promising in population genetics and seems extendable to more general settings, such as multi-type models.

Another important form of duality arises when the duality function is given by $H(x,y) = \mathds{1}_{\{x \leq y\}}$, with both $X$ and $Y$ valued in $[0,1]$. This is known as \emph{Siegmund duality} \cite{Siegmund1976}, which has been successfully applied in various contexts (e.g. \cite{Lo17,BLW16,CF21+,cordhumvech2022}). Unlike moment dualities, Siegmund duality holds under fairly general conditions. Specifically, it requires that the process $X$ is stochastically monotone with respect to its initial condition and satisfies some mild regularity assumptions (see e.g. \cite{Siegmund1976,Kolo11,Kolokoltsov2013}).

The notions of Siegmund and Bernstein duality are central to this work. We begin by showing that they are not entirely independent: when the process $X$ admits both a Siegmund and a Bernstein dual, we establish a simple connection between them. $\Lambda$-Wright--Fisher processes provide prototypical examples of such a setting (see \cite{cordhumvech2022,cordhumvech2024}). We then focus on $[0,1]$-valued ergodic Markov processes that admit either a Siegmund or a Bernstein dual. Our first application concerns convergence to stationarity. We provide quantitative bounds on the distance, in a suitable metric, between the distribution of $X_t$ and its stationary distribution. Thanks to duality, this typically difficult problem can be transformed into a simpler task involving the absorption properties of the dual process $Y$.

Our second application concerns recurrence properties of $X$. In contrast to the usual approach of deducing ergodicity from recurrence, we show that, in our setting, recurrence properties can be inferred from ergodicity. Specifically, under certain ergodicity assumptions, we prove that the support of the stationary distribution $\rho^X_\infty$ is topologically recurrent: any open neighbourhood of a point in the support is visited by the process $X$ at arbitrarily large times. This result provides a method for proving the existence of topologically recurrent points of $X$.

To illustrate our results, we apply them to a subclass of $\Lambda$-Wright--Fisher processes in population genetics. Specifically, we consider processes $X$ satisfying the stochastic differential equation
\begin{align}\label{eq:SDEWFP}
&\dd  X_t = \left( \sigma(X_t) X_t(1 - X_t) + \theta_a(1 - X_t) - \theta_A X_t \right)\dd t + \sqrt{\Lambda(\{0\}) X_t(1 - X_t)}\, \dd W_t \\
& + \int\limits_{(0,1] \times [0,1]} r \left( \mathds{1}_{\{u \leq X_{t-}\}}(1 - X_{t-}) - \mathds{1}_{\{u > X_{t-}\}} X_{t-} \right) \tilde{N}(\dd t, \dd r, \dd u) \nonumber \\
& + \int\limits_{(-1,1)} r X_{t-}(1 - X_{t-}) S(\dd t, \dd r) + \int\limits_{(-1,1)} |r| \left( \mathds{1}_{\{r \geq 0\}}(1 - X_{t-}) - \mathds{1}_{\{r < 0\}} X_{t-} \right) M(\dd t, \dd r), \nonumber
\end{align}
with initial condition $X_0 = x_0 \in [0,1]$. Here, $W \coloneqq (W_t)_{t \geq 0}$ is a standard Brownian motion, $\tilde{N}$ is a compensated Poisson measure on $[0,\infty) \times (0,1] \times [0,1]$ with intensity $\dd t \times r^{-2} \Lambda(\dd r) \times \dd u$, and $S$, $M$ are Poisson measures on $[0,\infty) \times (-1,1)$ with respective intensities $\dd t \times |r|^{-1} \mu(\dd r)$ and $\dd t \times |r|^{-1} \nu(\dd r)$. The processes $W$, $\tilde{N}$, $M$, and $S$ are assumed to be independent. The model is defined by the measures $\Lambda \in \Ms_f([0,1])$, $\mu, \nu \in \Ms_f([-1,1])$ (with $\mu(\{0\}) = \nu(\{0\}) = 0$), mutation rates $\theta = (\theta_a, \theta_A) \in \Rb_+^2$, and a frequency-dependent selection function $\sigma \in \Cs^1([0,1])$. The process $X$ models the proportion of individuals of type $a$ in a population subject to reproduction, selection, mutation, and environmental effects. For a detailed explanation of each term in \eqref{eq:SDEWFP}, we refer the reader to \cite{cordhumvech2024}.

It was shown in \cite{cordhumvech2024} that, under mutations and certain parameter conditions, the process $X$ is ergodic. In the absence of mutations, \cite{cordhumvech2022} identified parameter regimes in which $X$ is also ergodic. We apply our results to both cases to obtain bounds on the distance to stationarity and to derive simple, sharp conditions for open set recurrence of $[0,1]$, i.e., the property that starting from any point in $(0,1)$, the solution of \eqref{eq:SDEWFP} will visit any open subset of $[0,1]$ at arbitrarily large times.

Finally, we note that in many works on population genetics where $X$ is studied via moment duality (e.g. \cite{CorderoVechambre}), the moment dual is derived from an ancestral structure that can also be used to construct a Bernstein dual. Therefore, in such settings, Bernstein duality can be employed, and our results applied. This enables one to improve existing partial results on distance to stationarity and to investigate open set recurrence.

The paper is organized as follows. Section~\ref{S2} introduces the various notions of duality and states our main results. Section~\ref{bernsetinsection} contains the proofs related to Bernstein duality, including Theorems~\ref{thm:relduals}, \ref{propbounddtv}, \ref{thbounddiststa}, and Proposition~\ref{propcond}. Section~\ref{siegmundsection} presents the results concerning Siegmund duality: Theorems~\ref{genericsieg}, \ref{speedexpoofcvtostationarity}, and \ref{speedpolofcvtostationarity}. Section~\ref{sec:ergotoopensetrec} is independent from the others and addresses open set recurrence, proving Propositions~\ref{suppofmuinftyisrecunif}, \ref{suppofmuinftyisrec}, and Theorem~\ref{fullsupport}, which are then used in the aforementioned sections. Appendix~\ref{constructiondual} provides the construction of the Bernstein dual for solutions of \eqref{eq:SDEWFP}. Appendix~\ref{factslp} collects auxiliary results on L\'evy processes. Appendix~\ref{classicalfacts} reviews some standard metrics on probability measures.

\section{Preliminaries and main results }\label{S2}

In this section, we will review the concepts of Bernstein and Siegmund duality and present our key findings. Before doing so, let us introduce some notation.

\subsection{Notation} \label{notations}
For each $n,m\in\Nb$, we denote by $\Ls(\Rb^n,\Rb^m)$ the set of linear maps from $\Rb^n$ to $\Rb^m$. For any $C\in\Ls(\Rb^n,\Rb^m)$, we set $\lVert C\rVert_\infty\coloneqq\sup\{\lVert C w \rVert_\infty:w\in\Rb^n,\, \lVert w\rVert_\infty=1\}$.

For any $n\in\Nb$ and $p\in[0,1]$, we write $B\sim\bindist{n}{p}$ to indicate that $B$ is binomial random variable with parameters $n$ and $p$.

For any measure $\nu$ on $\mathbb{R}^d$, its support $Supp \ \nu$ is the set of points $x \in \mathbb{R}^d$ such that $\nu(\mathcal{O})>0$ for every open set $\mathcal{O}$ containing $x$. It is well-known that $(Supp \ \nu)^c$ is the largest open set with $\nu$-mass zero. In particular, every Borel set $A$ such that $\nu(A)>0$ contains at least one point of $Supp \ \nu$, otherwise it would be included in $(Supp \ \nu)^c$, which has $\nu$-mass zero. 

%%%%%%%%%%%%%%%%%%%%%%%%%%%%%%%%%%%%%%%%%%%%%%%%%%%%%%%%%%%%%%%%%%%%%
%%%%%%%%%%%%%%%%%%%%%%%%%%%%%%%%%%%%%%%%%%%%%%%%%%%%%%%%%%%%%%%%%%%%%
\subsection{Bernstein and Siegmund dualities}
We begin by revisiting the concept of Bernstein duality, which was introduced in \cite{Cordero2022}, and further generalized in \cite{Koske24,cordhumvech2024}.
\begin{defi}[Bernstein duality]\label{defberndual}
Let $X\coloneqq (X_t)_{t\geq 0}$ be a Markov process on $[0,1]$ and $V\coloneqq(V_t)_{t\geq 0}$ be a conservative Markov process on $\Rb^{\infty}\coloneqq\cup_{n\geq 0}\Rb^{n+1}$. We say that $X$ and $V$ are Bernstein duals if, for any $x\in[0,1]$, $v=(v_i)_{i=0}^n\in\R^{n+1}$ and~$t\geq0$,
\begin{align}
\E_x\left[\sum_{i=0}^n v_i \,\binom{n}{i} X_t^i(1-X_t)^{n-i}\right]=\E_v\left[\sum_{i=0}^{L_t} V_t(i)\,\binom{L_t}{i} x^i(1-x)^{L_t-i}\right], \label{bernsteinduality}
\end{align}
where $L_t\coloneqq \textrm{dim}(V_t)-1$. 
\end{defi}
We denote by $H:[0,1]\times \R^{\infty}$ the duality function which, according to the previous definition, is given by 
\begin{align}
	H(x,w)\coloneqq\sum_{i=0}^{\dim(w)-1}w_i \binom{\dim(w)-1}{i} x^i (1-x)^{\dim(w)-1-i}. \label{defh}
\end{align}
This allows us to rewrite the duality \eqref{bernsteinduality} as $\E_x\left[H(X_t,v)\right]=\E_v\left[H(x,V_t)\right]$.
Our results will not depend on the specific rates of the Bernstein dual, but it will heavily rely on the following assumption about its structure.
\begin{cond}\label{struct}
For any $n \geq 0$ and $v\in\R^{n+1}$, there is a process $(C_t)_{t \geq 0}$ of linear maps such that, $\Pb_v$-almost surely, for any $t\geq 0$, 
$$C_t \in \mathcal{L}(\R^{n+1},\R^{L_t+1}),\quad \lVert C_t \rVert_\infty \leq 1,\quad V_t = C_t v,$$
and, if $L_{t_0}=0$ for some $t_0\geq 0$, then $C_t=C_{t_0}$ for each $t\geq t_0$.
\end{cond}
Although the previous condition may not seem very intuitive, it is worth mentioning that it is satisfied in all previous works dealing with Bernstein duals \cite{Cordero2022, cordhumvech2024,Koske24}.  
Note that, under Condition \ref{struct}, \eqref{defh} yields, for any $v \in \Rb^{\infty}$, that $\P_v$-almost surely, for any $t\geq0$ and $x\in[0,1]$
\begin{align}
 \ |H(x,V_t)| \leq \lVert V_t \rVert_\infty = \lVert C_tv \rVert_\infty \leq ||v||_{\infty}. \label{normineq}
\end{align}

We now move to the notion of Siegmund duality.
\begin{defi}[Siegmund duality]Let $X\coloneqq (X_t)_{t\geq 0}$ and $Y\coloneqq(Y_t)_{t\geq 0}$ Markov processes on $[0,1]$. We say that $X$ and $Y$ are Siegmund duals if, for any $x,y\in[0,1]$, and~$t\geq0$,
\begin{align}
\P_x\left(X_t\geq y\right)=\P_y\left(x\geq Y_t\right). \label{siegmundduality}
\end{align}
\end{defi}
In his seminal work, Siegmund \cite{Siegmund1976} showed that a real-valued Markov process for which $\infty$ (when reachable) is an absorbing state admits a Siegmund dual if and only if the process is stochastically monotone and right-continuous, in the sense that the function $x \mapsto \Pb_x(X_t \geq y)$ is right-continuous for all $y \in [0,1]$ and $t \geq 0$. Since we focus on $[0,1]$-valued processes, no condition at $\infty$ is required, and stochastic monotonicity together with regularity with respect to the initial condition suffices. However, some of our results require a slightly stronger assumption than stochastic monotonicity; namely, we will need to assume the following condition is satisfied.
\begin{cond}\label{moncoup}
For any $x_1,x_2\in[0,1]$ with $x_1\leq x_2$ (resp. $x_1\geq x_2$) there is a $[0,1]^2$-valued process $(X_t^1,X_t^2)_{t\geq 0}$ with $X_t^i\sim\Pb_{x_i}$, $i\in\{1,2\}$, and $X_t^1\leq X_t^2$ (resp. $X_t^1\geq X_t^2$) for all $t\geq 0$ almost surely. We refer to $(X_t^1,X_t^2)_{t\geq 0}$ as a $(x_1,x_2)$-monotone coupling and denote its law by $\Pb_{x_1,x_2}$. Moreover, the map $(x_1,x_2)\mapsto \Pb_{x_1,x_2}(\cdot)$ is measurable. 
\end{cond}

The motivation behind Condition~\ref{moncoup} is that, when combined with Siegmund duality, it allows one to show - following the approach used in the proof of \cite[Cor.~2.6]{cordhumvech2022} - that for any $x_1, x_2, y \in [0,1]$ with $x_1 < x_2$, and for any $(x_1, x_2)$-monotone coupling $(X^1_t, X^2_t)_{t \geq 0}$, we have
\begin{eqnarray}
\mathbb{P}_{x_1,x_2} \left ( X^1_t < y \leq X^2_t \right ) = \mathbb{P}_{y} \left ( x_1 < Y_t \leq x_2 \right ). \label{doubleduality}
\end{eqnarray}
This identity will prove useful to study ergodicity of $X$ through its Siegmund dual $Y$. The existence of measurable monotone couplings for right-continuous, stochastically monotone processes on $\Rb$ holds under mild additional assumptions. For example, for processes with continuous paths, such a coupling can be constructed by starting with two independent copies of the process, running them until their first meeting time, and letting them evolve together thereafter. In the case of continuous-time Markov chains, existence follows from Theorem 1 in \cite{CliffSud85}. The existence of monotone couplings for processes defined as solutions to SDE's is typically obtained from standard comparison results (see e.g. Theorem 2.3 in \cite{DaLi12}); the measurability with respect to the initial pair $(x_1,x_2)$ requires in general additional work. Moreover, the recent work \cite{BeFr24} shows that the Feller property is sufficient for the existence of measurable monotone couplings in stochastically monotone processes. This result is particularly suitable for our applications to $\Lambda$-Wright--Fisher processes. 
\subsection{Connecting Bernstein and Siegmund dualities}

We now demonstrate that these two seemingly unrelated concepts are, in fact, closely connected. To establish this relationship, we begin by introducing some notation. For a fixed $x \in [0,1]$, let $(\beta_x(j))_{j \geq 1}$ be a sequence of random variables such that $\beta_x(j) \sim \bindist{j}{x}$ for each $j \geq 1$, and assume that these variables are mutually independent and also independent of the process $(V_t)_{t \geq 0}$. We continue to denote by $\P_v$ the probability measure on the space where both $(V_t)_{t \geq 0}$ and $(\beta_x(j))_{j \geq 1}$ are defined. 
The following result shows that, for a process admitting both a Bernstein dual and a Siegmund dual, the initial distributions of the duals can be coupled in such a way that the distribution function of the Siegmund dual at time $t$ is uniformly approximated by the expectation of certain randomized coefficients of the Bernstein dual at time $t$.

\begin{theo}\label{thm:relduals}
Let $X\coloneqq (X_t)_{t\geq 0}$ be a Markov process on $[0,1]$. Assume that $X$ admits a Bernstein dual $V$ and a Siegmund dual $Y$. Let $F:[0,1] \rightarrow [0,1]$ be a distribution function such that $F \in \mathcal{C}^1([0,1])$. We denote by $\mathbb{P}_F$ the law of the process $Y$ with initial distribution $Y_0 \sim F$. Define $v_{F,n}\coloneqq(F(i/n))_{i=0}^n\in\R^{n+1}$. For all $n\geq 1$ we have 
\begin{align}
\sup_{x \in [0,1], t \geq 0} \left | \mathbb{P}_F(Y_t \leq x) - \E_{v_{F,n}}[V_t(\beta_x(L_t))] \right | \leq \frac{||F'||_{\infty}}{2 \sqrt{n}}. \label{approxftbygt}
\end{align}
\end{theo}
Beyond its theoretical significance in connecting two distinct concepts of duality, Theorem \ref{thm:relduals} has a potential practical application: if a Markov process on $[0,1]$ possesses a Siegmund dual that, in turn, admits a Bernstein dual, then this Bernstein dual can be used to uniformly approximate the distribution function of the original process.
%\begin{remark}
%We consider the $(\Lambda,\mu,\nu,\beta,p)$-Bernstein coefficient process $(V_t)_{t\geq 0}$ with the initial condition $V_0=v$. If we randomly draw $\ell$ balls without replacement in an urn containing $n$ balls, $i$ being of type $a$ and $n-i$ being of type $A$, then $v_i$ is to the probability of the $\ell$ drawn balls being of type $a$. An interpretation of $V_t(i)$ (for $i \in \{1,\ldots,L_t\}$) is given via the following procedure: at time $t$ we choose randomly $n$ individuals in the population and (randomly) mark $\ell$ of them. Then we trace back the ancestors of the $n$ individuals via the ASG, assign types randomly to the $L_t$ ancestors from time $0$ (with the constraint of exactly $i$ individuals being assigned type $a$) and propagate the types forward in time in the ASG according to the rules of Section \ref{sec:asg}. Then $V_t(i)$ is interpreted as the probability, conditionally on the ASG, that the $\ell$ marked individuals from time $t$ receive type $a$ at the end of this procedure. The following result shows that $\E_v[V_t(i)]$ (for some $i$ such that $i \approx xL_t$) approximates the distribution function of $Y_t$ at $x$ when $n$ is large.
%\end{remark}
\subsection{Distance to stationarity for ergodic processes with a dual} \label{criteriaforergodicity}

In this section, we study the ergodic properties of a Markov process $X$ on $[0,1]$, which admits a Bernstein dual $V$ or a Siegmund dual $Y$. In this context, we derive bounds for the distance to stationarity, i.e. the distance (in a suitable metric) between the process's law at time $t$ and its stationary distribution. In some way, the two results in this section (Theorems \ref{propbounddtv} and \ref{genericsieg}) transform the difficult problem of studying the proximity between two coupled trajectories of the Markov process $X$ into the simpler problem of studying the absorption of a single trajectory of a dual process. To make this precise, we first introduce some notions of ergodicity.

\begin{defi}[Ergodicity notions]
Let $X$ be a Markov process on $[0,1]$ and denote by $\rho^X_{x,t}$ the law of $X_t$ under $\mathbb{P}_x$, i.e $\rho^X_{x,t}(\cdot)\coloneqq\mathbb{P}_x(X_t \in \cdot)$. We say that $X$ is \emph{ergodic} on $(0,1)$ if there is a stationary distribution $\rho^X_{\infty}$ on $[0,1]$ such that, for all $x \in (0,1)$, $\rho^X_{x,t}$ converges weakly to $\rho^X_{\infty}$ as $t \rightarrow \infty$. 

Let $d$ be metric on the set of probability measures on $[0,1]$. We say that $X$ is \emph{locally $d$-uniformly ergodic} on $(0,1)$ if there is a stationary distribution $\rho^X_{\infty}$ such that, for each $\epsilon\in(0,1/2)$, 
\begin{align*}
 \ \sup_{x \in [\epsilon, 1-\epsilon]} d(\rho_{x,t}^X,\rho^X_{\infty}) \underset{t \rightarrow \infty}{\longrightarrow} 0. 
\end{align*}
We say that $X$ is \emph{$d$-uniformly ergodic} on $[0,1]$ if there is a stationary distribution $\rho^X_{\infty}$ such that 
\begin{align*}
\sup_{x \in [0, 1]} d(\rho_{x,t}^X,\rho^X_{\infty}) \underset{t \rightarrow \infty}{\longrightarrow} 0. 
\end{align*}
\end{defi}

Clearly, $d$-uniform ergodicity on $[0,1]$ implies locally $d$-uniform ergodicity on $(0,1)$. Moreover, if $d$ is stronger than the L\'evy-Prokhorov metric $d_{LP}$, locally $d$-uniform ergodicity on $(0,1)$ implies ergodicity on $(0,1)$. Our results will involve the $W_p$-Wasserstein distances and the Radon distance $d_{Rad}$. The definitions and basic properties of these metrics are recalled in Appendix \ref{classicalfacts} (in particular, they are stronger than $d_{LP}$). 

Since our results will involve absorption properties of the dual process, we introduce appropriate notions of absorption.
\begin{defi}[Absorption notions]
We say that a process $(Z_t)_{t\geq 0}$ in $[0,1]$ is \textit{asymptotically absorbed} in $\{0,1\}$ if, for any $z \in [0,1]$, $\mathbb{P}_z$-almost surely the limit $Z_{\infty}\coloneqq\lim_{t \rightarrow \infty} Z_t$ exists and belongs to $\{0,1\}$. We say that such a process is \textit{regularly absorbed} in $\{0,1\}$ if 
\begin{align*}
\lim_{z \rightarrow 1} \mathbb{P}_z (Z_{\infty}=1)=1 \ \text{and} \ \lim_{z \rightarrow 0} \mathbb{P}_z (Z_{\infty}=0)=1.
\end{align*}
We say that it is \textit{uniformly absorbed} in $\{0,1\}$ if for any $\epsilon\in(0,1/2)$
\begin{align*}
 \ \sup_{z \in [0,1]} \mathbb{P}_z \left ( Z_t\in [\epsilon,1-\epsilon] \right ) \underset{t \rightarrow \infty}{\longrightarrow} 0. 
\end{align*}

\end{defi}

The next result establishes, under some conditions, the uniform ergodicity (for the distances $W_1$ and $d_{Rad}$) of $[0,1]$-valued processes admitting a Bernstein dual. 
\begin{theo}[Bernstein dual case] \label{propbounddtv}
Let $X\coloneqq (X_t)_{t\geq 0}$ be a Markov process in $[0,1]$. Assume that $X$ admits a Bernstein dual $V$ satisfying Condition \ref{struct}. Set $L_t\coloneqq \textrm{dim}(V_t)-1$, $t\geq 0$, and assume  that $(L_t)_{t\geq 0}$ is absorbed at $0$ in finite time almost surely. Then $X$ is $W_1$-uniformly ergodic on $[0,1]$ and, denoting by $\rho^X_{\infty}$ its stationary distribution, the following estimates hold. 
\begin{enumerate}
\item For any $t \geq 0$, 
\begin{align}
2 \sup_{x \in [0,1]} W_1 (\rho_{x,t}^X, \rho^X_{\infty}) \leq \sup_{x \in [0,1]} d_{Rad} (\rho_{x,t}^X, \rho^X_{\infty}) \leq 2 \left ( \sup_{n \geq 1} \mathbb{P}_n(L_t \neq 0) \right ). \label{eqboundradon}
\end{align}
\item For any $n\geq 1$ and $t \geq 0$, 
\begin{align}
\sup_{x \in [0,1]} W_1 (\rho_{x,t}^X, \rho^X_{\infty}) \leq \frac{1}{\sqrt{n}} + \mathbb{P}_n(L_t \neq 0). \label{eqboundwass}
\end{align}
\end{enumerate}
\end{theo}
Thanks to \eqref{classicalLP}, we see that, under the assumptions in Theorem \ref{propbounddtv}, the process $X$ is also $d_{LP}$-uniformly ergodic on $[0,1]$. 

The next result establishes, under some conditions, the locally uniform ergodicity (for the distances $W_p$ with $p\geq 1$) of $[0,1]$-valued processes admitting a Siegmund dual. 
\begin{theo}[Siegmund dual case] \label{genericsieg}
Let $X\coloneqq (X_t)_{t\geq 0}$ be a Markov process in $[0,1]$ satisfying Condition \ref{moncoup}. In addition, assume that $X$ admits a Siegmund dual $Y$, which  is regularly and uniformly absorbed in $\{0,1\}$. Then, $X$ is locally $W_p$-uniformly ergodic on $(0,1)$ for all $p\geq 1$. Moreover, denoting by $\rho^X_{\infty}$ its stationary distribution, we have $\rho^X_{\infty}(\{0,1\})=0$ and, for any $\epsilon \in(0,1/2)$ and $t\geq 0$, 
\begin{align}
\sup_{x\in [\epsilon,1-\epsilon]} W_p(\rho^X_{x,t},\rho^X_{\infty})^p \leq \sup_{y \in [0,1]} \mathbb{P}_y \left ( Y_t\in [\epsilon,1-\epsilon] \right ) + \mathbb{P}_{\epsilon} (Y_{\infty}=1) + \mathbb{P}_{1-\epsilon} (Y_{\infty}=0). \label{decaywpsiegmund}
\end{align}
\end{theo}
As before, thanks to \eqref{classicalLP}, we see that, under the assumptions in Theorem \ref{genericsieg}, the process $X$ is also locally $d_{LP}$-uniformly ergodic on $(0,1)$. 

\subsection{Open set recurrence}\label{sec:opensetrec}

When a process is ergodic, it is also of interest to determine whether it is open-set recurrent. Often, the ergodicity properties of a Markov process can be deduced from its recurrence properties. However, in our case, our duality-based methodology provides more insight into the ergodicity properties of the process than its recurrence properties. We therefore present results that enable the latter to be deduced from the former. We start introducing the notions of recurrence that we use.
\begin{defi}[topological recurrence]
We said that a point $x_0\in[0,1]$ is \textit{topologically recurrent} for $X$ if any of the open neighbourhoods of $x_0$ is visited by $X$ at arbitrary large times almost surely. The set of topologically recurrent points is then 
\begin{align*}
\mathcal{R}(X)\coloneqq \big \{ x \in [0,1]: \ \forall z \in (0,1), \forall \eta>0, \mathbb{P}_z \big ( \limsup_{t \rightarrow \infty} \mathds{1}_{\{X_t \in (x-\eta,x+\eta)\}}=1 \big ) = 1 \big \}. 
\end{align*}
Note that $\mathcal{R}(X)$ is a closed set. We say that $X$ is \textit{open set recurrent} if $\mathcal{R}(X)=[0,1]$. In other words, $X$ is open set recurrent if, for every starting point $z \in (0,1)$, each open set is $\P_z$-almost surely visited by $X$ at arbitrary large times.
\end{defi}

 The main results of this section establish conditions, expressed in terms of the ergodicity properties of $X$, which imply that any point in $Supp \ \rho^X_{\infty}$ is topologically recurrent.
\begin{prop}[uniform ergodicity condition] \label{suppofmuinftyisrecunif}
Let $X$ be a Markov process on $[0,1]$. Assume that $X$ is $d_{LP}$-uniformly ergodic on $[0,1]$ and denote its stationary distribution by $\rho^X_{\infty}$. Then $Supp \ \rho^X_{\infty} \subset \mathcal{R}(X)$. 
\end{prop}

\begin{prop}[local uniform ergodicity condition] \label{suppofmuinftyisrec}
Let $X$ be a Markov process on $[0,1]$. Assume that $X$ is locally $d_{LP}$-uniformly ergodic on $(0,1)$ and denote its stationary distribution by $\rho^X_{\infty}$. We also assume that $\rho^X_{\infty}(\{0,1\})=0$. Then $Supp \ \rho^X_{\infty} \subset \mathcal{R}(X)$. 
\end{prop}

Propositions \ref{suppofmuinftyisrecunif} and \ref{suppofmuinftyisrec} provide conditions for a Markov process $X$ to admit at least one topologically recurrent point. The following result completes them by providing a simple and sharp condition for a $\Lambda$-Wright--Fisher process of the form given by \eqref{eq:SDEWFP}, having at least one topologically recurrent point in $(0,1)$, to be open set recurrent.

\begin{theo} \label{fullsupport}
Let $X$ be the solution of \eqref{eq:SDEWFP}. Assume that $\Lambda([0,1/2+\epsilon])>0$ for every $\epsilon>0$ and $\mathcal{R}(X) \cap (0,1) \neq \emptyset$. Then $\mathcal{R}(X)=[0,1]$, and in particular $X$ is open set recurrent. 
\end{theo}

Typical choices of $\Lambda$ in mathematical population genetics are $\Lambda(\dd x)=\delta_0(\dd x)$ and $\Lambda(\dd x)=cx^{a-1}(1-x)^{b-1}\dd x$, with $a,b,c>0$. In all these cases, the condition imposed on $\Lambda$ in Theorem \ref{fullsupport} is satisfied. Let us also mention that the proof of Theorem \ref{fullsupport} treats separately the case "$\Lambda(\{0\})>0$" (i.e. when $X$ has a diffusion component) from the case "$\Lambda(\{0\})=0$ and $\Lambda((0,1/2+\epsilon])>0$ for every $\epsilon>0$", with each case requiring a different technique.

\begin{remark} \label{remarkoptimality}
The "$1/2 +\epsilon$", in the condition on the measure $\Lambda$ in Theorem \ref{fullsupport}, is optimal in the sense that, for any $\Lambda\in\Ms_f([0,1])$ that does not satisfy the condition, we can choose the remaining parameters defining $X$ such that $\mathcal{R}(X) \cap (0,1) \neq \emptyset$ and $1/2 \notin \mathcal{R}(X)$. These counter-examples are provided in Section \ref{optimalityofcondition}. 
\end{remark}

A combination of the results presented in this section with those from Section \ref{criteriaforergodicity} will allow us to identify a class of $\Lambda$-Wright--Fisher processes that are open set recurrent.
\subsection{Applications to $\Lambda$-Wright--Fisher processes}\label{sec:applications}
In this section we will apply the results of Sections \ref{criteriaforergodicity} and \ref{sec:opensetrec} to two classes of ergodic $\Lambda$-Wright--Fisher processes of the form given by \eqref{eq:SDEWFP}. The first class consists of $\Lambda$-Wright--Fisher processes subject to bidirectional mutations, whose parameters satisfy a condition that, according to the results in \cite{cordhumvech2024}, guarantees its ergodicity. The second class comprises $\Lambda$-Wright--Fisher processes without mutation,  with moderate neutral reproduction and frequency-dependent selection with a sufficiently strong repelling effect at the boundaries, rendering the process ergodic. The precise condition on the parameters stems from \cite{cordhumvech2022}.

Before we dive into the analysis of these two classes of processes, let us recall that the existence and pathwise uniqueness of strong solutions of SDE \eqref{eq:SDEWFP} can be shown analogously to \cite[Lemma 3.2]{Cordero2022} (see also \cite[Prop. A.3]{cordhumvech2022}).

\subsubsection{$\Lambda$-Wright--Fisher processes with mutation}\label{sec:model}
In this section we consider the first of the two aforementioned classes of ergodic $\Lambda$-Wright--Fisher processes. The first assumption is that the mutation parameters satisfy
\begin{align}
	\theta_a+\nu(0,1)>0 \text{ and } \theta_A+\nu(-1,0)>0, \label{eq:bimut}
\end{align}
which biologically means that the process $X$ is subject to bidirectional mutations. The second condition involves the parameter of frequency-dependent selection $\sigma$, and ensures the existence of a Bernstein dual \cite{cordhumvech2024}. More precisely, we assume that $\sigma$ is a polynomial. In this case, according to \cite{Cordero2022}, there is an integer $\kappa>1$, a vector of rates $\beta\coloneqq(\beta_\ell)_{\ell\in]\kappa]}\in \R_{+}^{\kappa-1}$, and a collection of probabilities
$p\coloneqq (p_i^{(\ell)}:\ell\in ]\kappa], i\in[\ell]_0)$ with $p^{(\ell)}_i\in[0,1]$ and $p_0^{(\ell)}=0=1-p_\ell^{(\ell)}$, such that
\begin{align}
\sigma(x)x(1-x)=\sum_{\ell=2}^{\kappa}\beta_\ell\sum_{i=0}^\ell \binom{\ell}{i}x^i(1-x)^{\ell-i}\Big(p_{i}^{(\ell)}-\frac{i}{\ell}\Big). \label{defselectionterm}
\end{align}
This representation of $\sigma$ is not unique (see \cite{Cordero2022}), and therefore, we set $s\coloneqq(\kappa,\beta,p)$ and write $\sigma_s$ instead of $\sigma$ to emphasize the representation of $\sigma$ we are using.

In \cite[Thm. 2.2]{cordhumvech2024} it was proved that a solution $X$ of SDE \eqref{eq:SDEWFP} with $\sigma=\sigma_s$ has a Bernstein dual. Let us mention that the choice of $s$ has an effect on the dynamics of the dual. A construction of the dual process is given in Appendix \ref{constructiondual}. 

\begin{theo} \label{thbounddiststa}
Let $X\coloneqq (X_t)_{t\geq 0}$ be a solution of SDE \eqref{eq:SDEWFP} with $\sigma=\sigma_s$. Assume that
\begin{align}
C \coloneqq \nu(-1,1) + \theta_a + \theta_A - \mu(-1,1) - \sum_{\ell=2}^{\kappa} \beta_\ell (\ell-1)>0. \label{exporate}
\end{align}
Then $X$ is $W_1$-uniformly ergodic on $[0,1]$ and, if $\rho^X_{\infty}$ denotes the stationary distribution of $X$, the following bounds hold.
\begin{enumerate}
\item For any $t\geq 0$ we have 
\begin{align}
\sup_{x \in [0,1]} W_1 (\rho_{x,t}^X, \rho^X_{\infty}) \leq 3 e^{-Ct/3}. \label{finalboundwass}
\end{align}
\item Further, if $\sup_{n\in\Nb}\Eb_n[L_1]<\infty$, then there is $K>0$ such that for any $t\geq 1$, 
\begin{align}
2 \sup_{x \in [0,1]} W_1 (\rho_{x,t}^X, \rho^X_{\infty}) \leq \sup_{x \in [0,1]} d_{Rad} (\rho_{x,t}^X, \rho^X_{\infty}) \leq 2 K e^{-Ct}. \label{finalboundradon}
\end{align}
\end{enumerate}
 If, in addition, $\Lambda([0,1/2+\epsilon])>0$ for every $\epsilon>0$, then $X$ is also open set recurrent.
\end{theo}
The bound \eqref{finalboundradon}, which slightly improves \eqref{finalboundwass}, requires the generic condition $\sup_{n\in\Nb}\Eb_n[L_1]<\infty$. We close this section with a simple condition on the measure $\Lambda$ that implies it. More precisely, we will assume that
\begin{align}
\limsup_{x \rightarrow 0} \frac{\log \Lambda([0,x))}{\log(x)}<1. \label{thecondition}
\end{align}

\begin{prop} \label{propcond}
For any $\Lambda\in\Ms_f([0,1])$ satisfying \eqref{thecondition}, the underlying process $L$ satisfies  $$\sup_{n \geq 1} \mathbb{E}_n[L_1] < \infty.$$ 
\end{prop}

\begin{remark}\label{exco}
If $\Lambda(\{0\})>0$, then the limsup in \eqref{thecondition} equals $0$, and hence \eqref{thecondition} is satisfied. Another classical choice for the measure $\Lambda$ is the Beta distribution, i.e. $\Lambda(\dd r)\coloneqq c \ r^{a-1}(1-r)^{b-1}\dd r$ for some $a,b,c>0$. In this case the limsup in \eqref{thecondition} equals $a$, thus \eqref{thecondition} is satisfied if and only if $a \in (0,1)$ and $b, c>0$. 
\end{remark}

\begin{remark}
If $\int_{[0,1]} z^{-1} \Lambda(\dd z)<\infty$ (known as \textit{dust condition}), then $\Lambda([0,x)) \leq x \int_{[0,1]} z^{-1} \Lambda(\dd z)$. Thus, the limsup in \eqref{thecondition} is larger or equal to $1$ and \eqref{thecondition} does not hold. In particular,  \eqref{thecondition} contains the assumption $\int_{[0,1]} z^{-1} \Lambda(\dd z)=\infty$ (known as \textit{no dust condition}). 

In the Beta case, i.e. $\Lambda(\dd r)\coloneqq c \ r^{a-1}(1-r)^{b-1}\dd r$, the condition $\int_{[0,1]} z^{-1} \Lambda(\dd z)=\infty$ is satisfied if and only if $a \in (0,1]$ and $b, c>0$. This together with Remark \ref{exco} shows that \eqref{thecondition} is slightly, but strictly stronger than $\int_{[0,1]} z^{-1} \Lambda(\dd z)=\infty$. 
\end{remark}

\subsubsection{$\Lambda$-Wright--Fisher processes without mutation} \label{applsiegy}
Now, we consider the second class of ergodic $\Lambda$-Wright--Fisher processes. The first assumption on the parameters is that 
\begin{align}
\int_{[0,1]}r^{-1}\Lambda(\dd r)\in(0,\infty),\qquad \Lambda(\{1\})=0,\qquad \nu=0,\qquad \theta_a=\theta_A=0. \label{assumptionnomut}
\end{align}
The first condition imposes a restriction on the strength of neutral reproductions; the conditions on $\nu,\theta_a,\theta_A$ mean that $X$ is not subject to mutations. In \cite[Thm. 2.5]{cordhumvech2022} it was proved that $X$ admits a Siegmund dual.
To determine the long-term behaviour of $X$, one has to compare the neutral forces, parametrized by~$\Lambda$, with the selective forces, parametrized by~$\mu$ and~$\sigma$. To do this, define \begin{align*}
	 C_0(\Lambda,\mu,\sigma)&\textstyle \coloneqq \left(\sigma(0)- \int_{(-1,1)}\log\big(\frac{1}{1+r}\big) {\mu}(\dd r) \right)-\int_{(0,1)}\log\big(\frac{1}{1-r}\big) \frac{\Lambda(\dd r)}{r^2},\\ 
	C_1(\Lambda,\mu,\sigma)&\textstyle \coloneqq \left(-\sigma(1) -\int_{(-1,1)}\log\big(\frac{1}{1-r}\big) \mu(\dd r)\right)-\int_{(0,1)}\log\big(\frac{1}{1-r}\big) \frac{\Lambda(\dd r)}{r^2},
\end{align*}
We know from \cite[Thm. 2.1(3)]{cordhumvech2022} that, in the case where 
\begin{align}
C_0(\Lambda,\mu,\sigma)>0 \ \text{and} \ C_1(\Lambda,\mu,\sigma)>0, \label{assumptionc0c1}
\end{align}
then the process $X$ is ergodic and has a unique stationary distribution $\rho^X_{\infty}$ supported on $(0,1)$. Moreover, $X_t$ converges in distribution as $t$ goes to infinity to $\rho^X_{\infty}$, for every starting point $x \in (0,1)$. This parameter regime was referred to as \emph{coexistence regime}. The aim of this section is to study the distance to stationarity in this regime. 

We will assume two types of integrability for the big jumps, namely, the weak integrability conditions 
\begin{align}\tag{$WH_{\gamma}$}\label{weakhyp}
&\int\limits_{(\frac12,1)} \!\!\log \left ( \frac1{1-r} \right ) ^{1+\gamma} \!\!  \Lambda(\dd r) ,\,\int\limits_{(\frac12,1)} \!\! \log \left ( \frac1{1-r} \right ) ^{1+\gamma}  \!\! \mu(\dd r), \, \int\limits_{(-1,-\frac12)} \!\!\! \log \left ( \frac1{1+r} \right )^{1+\gamma}  \!\! \mu(\dd r) < \infty. \nonumber
\end{align}
and the strong integrability conditions 
\begin{align}\tag{$SH_{\gamma}$}\label{stronghyp}
\int_{(\frac12,1)} \frac{\Lambda(\dd r)}{(1-r)^{\gamma}} 
,\, \ \int_{(\frac12,1)} \frac{\mu(\dd r)}{(1-r)^{\gamma}}, \,
 \ \int_{(-1,-\frac12)} \frac{\mu(\dd r)}{(1+r)^{\gamma}} < \infty. 
\end{align}
We end this section with the following results, which state the ergodicity properties of $X$ under the strong and weak integrability conditions, respectively. 
\begin{theo} \label{speedexpoofcvtostationarity}
We assume \eqref{assumptionnomut} and \eqref{assumptionc0c1} and that \eqref{stronghyp} is satisfied for some $\gamma>0$. Then, for any $p\geq 1$, $X$ is locally $W_p$-uniformly ergodic on $(0,1)$. Let us denote by $\rho^X_{\infty}$ the stationary distribution of $X$. Then, for any $p \geq 1$ and $\epsilon \in (0,1/2)$, there are positive constants $c_1 = c_1(p,\epsilon)$ and $c_2 = c_2(p)$ such that, for any $t \geq 0$,  
\begin{eqnarray}
\sup_{x \in [\epsilon, 1-\epsilon]} W_p(\rho^X_{x,t},\rho^X_{\infty}) \leq c_1 e^{-c_2 t}. \label{boundexpowasserstein}
\end{eqnarray}
If, in addition, $\Lambda([0,1/2+\epsilon])>0$ for every $\epsilon>0$, then $X$ is open set recurrent.
\end{theo}

\begin{theo} \label{speedpolofcvtostationarity}
We assume \eqref{assumptionnomut} and \eqref{assumptionc0c1} and that \eqref{weakhyp} is satisfied for some $\gamma>0$. Then, for any $p\geq 1$, $X$ is locally $W_p$-uniformly ergodic on $(0,1)$. Let us denote by $\rho^X_{\infty}$ the stationary distribution of $X$. Then, for any $p \geq 1$, $\epsilon \in (0,1/2)$ and $\beta \in (0,(1+\gamma)\gamma/(2\gamma+1)p)$, there is a positive constant $C = C(p,\epsilon,\beta)$ such that, for any $t \geq 0$, 
\begin{eqnarray}
\sup_{x \in [\epsilon, 1-\epsilon]} W_p(\rho^X_{x,t},\rho^X_{\infty}) \leq C t^{-\beta}. \label{boundpolwasserstein}
\end{eqnarray}
If, in addition, $\Lambda([0,1/2+\epsilon])>0$ for every $\epsilon>0$, then $X$ is open set recurrent.
\end{theo}

\section{Proofs of results for processes with Bernstein dual} \label{bernsetinsection}
We begin this section by revisiting some approximation results from real analysis. Let $f \in \mathcal{C}([0,1])$ and let $f_n$ be the Bernstein polynomial approximation of $f$, i.e. $$f_n(z)\coloneqq\sum_{i=0}^n f(i/n)\binom{n}{i} z^i(1-z)^{n-i}.$$ By Bernstein uniform convergence theorem, we have 
\begin{align}
|| f - f_n ||_{\infty} \underset{n \rightarrow \infty}{\longrightarrow} 0. \label{bernunifcvth}
\end{align}
Moreover, if $f$ belongs to the set $\mathcal{L}_c([0,1])$ of Lipschitz continuous functions on $[0,1]$ with Lipschitz constant $c$, then it is well-known that
\begin{align}
 \ || f - f_n ||_{\infty} \leq \frac{c}{2 \sqrt{n}},\qquad n\in\Nb. \label{bernunifcvthlip}
\end{align}
Let us assume now that $X$ is a $[0,1]$-valued Markov process that admits a Bernstein dual $V$ satisfying Condition \ref{struct}.
Fix now $x \in [0,1]$, $n \geq 1$ and $t \geq 0$. Using the duality relation \eqref{bernsteinduality} with $v\coloneqq v_{f,n}\coloneqq (f(i/n))_{i=0}^n\in\R^{n+1}$, we get, for $t\geq0$,
\begin{align}
%\mathbb{E}_x\left[f_n(X_t)\right]=\mathbb{E}_{v_{f,n}}\left[\sum_{i=0}^{L_t} V_t(i)\,\binom{L_t}{i} x^i(1-x)^{L_t-i}\right]. \label{berndualfn}
\mathbb{E}_x\left[f_n(X_t)\right]=\mathbb{E}_{v_{f,n}}\left[H(x,V_t)\right]. \label{berndualfn}
\end{align}
In Sections \ref{proofrelation} and \ref{proofbernsteincase} we will use this expression combined with \eqref{bernunifcvth}-\eqref{bernunifcvthlip} to approximate quantities of the form $\mathbb{E}_x[f(X_t)]$.

\subsection{Relation between Bernstein and Siegmund duals} \label{proofrelation}
This section is devoted to proving Theorem \ref{thm:relduals}. Recall that the process $X$ in the statement has both a Siegmund dual $Y$ and a Bernstein dual $V$.
\begin{proof}[Proof of Theorem \ref{thm:relduals}]
Let $F$ be as in the statement of the theorem. In particular, $F \in \mathcal{L}_c([0,1])$ with $c=||F'||_{\infty}$.  Let also $Z$ be a random variable with distribution function $F$, independent of $X$. By slight abuse of notation, we continue to denote by $\P_x$ the probability measure in the probability space on which $X$ and $Z$ are defined. We start with the observation that the right hand-side in the Bernstein duality \eqref{bernsteinduality} can be rewritten as 
\begin{align}
\E_v\left[H(x,V_t)\right]=\E_v\left[V_t(\beta_x(L_t))\right], \label{berndualbin}
\end{align}
where $\beta_x(\cdot)$ is as in the statement of the theorem. Using the definitions of $Z$ and the Siegmund duality we get 
\begin{align}
\E_x\left[F(X_t)\right] = \mathbb{P}_{x}(Z \leq X_t) = \mathbb{P}_F(Y_t \leq x).  \label{momxtofctrepy}
\end{align}
Fix $n \geq 1$ and set $v=v_{F,n}$, where $v_{F,n}\in\R^{n+1}$ is as in the statement of the theorem. 
The combination of \eqref{berndualbin}, \eqref{berndualfn}, \eqref{bernunifcvthlip} and \eqref{momxtofctrepy} yields \eqref{approxftbygt}. 
\end{proof}
\subsection{Distance to stationarity via Bernstein duality} \label{proofbernsteincase}
In this section, we prove Theorem \ref{propbounddtv} concerning the ergodic properties of $[0,1]$-valued processes, which admit a Bernstein dual satisfying Condition \ref{struct} and which become trapped in $\Rb$ in finite time.

\begin{proof}[Proof of Theorem \ref{propbounddtv}]
The combination of the assumption on the absorption of $L$ at $0$ with Condition \ref{struct} yields that, for any $v \in \R^\infty$, $\P_{v}$-almost surely $$\tau \coloneqq \inf \{ s \geq 0: L_s=0 \}<\infty,$$  and there is a random variable $U_\infty$ such that, for $t \geq \tau$, we have $\dim(V_t)=1$ and $V_t$ is the vector whose single entry is $U_\infty$. In particular, for any $x \in [0,1]$, $H(x,V_t)=U_\infty$ for all $t \geq \tau$. Let $e_k \in \R^{k+1}$ be given by $e_k(i)\coloneqq \mathds{1}_{\{i=k\}}$. Thanks to Bernstein duality \eqref{bernsteinduality} we get  $$\E_x[X_t^k] = \E_x[H(X_t,e_k)] = \E_{e_k}[ H(x,V_t)].$$ By \eqref{normineq} we have $\P_{e_k}$-a.s. $|H(x,V_t)|\leq 1$ for all $t\geq 0$. We thus get by dominated convergence that $\E_{e_k}[ H(x,V_t)]$, and therefore $\E_x[X_t^k]$, converges to $\E_{e_k}[ U_\infty]$ as $t\to\infty$. Therefore, all positive integer moments of $X_t$ converge as $t\to\infty$ to a limit, which does not depend on the initial condition $x$. Since probability distributions on [0, 1] are fully characterised by their positive integer moments and the convergence of these moments implies convergence in distribution, we conclude that, there is a distribution $\rho^X_{\infty}$ on $[0,1]$ such that, for any $x \in (0,1)$, the law of $X_t$ under $\P_x$ converges weakly to $\rho^X_{\infty}$. Moreover, for any $k\geq 1$, 
\begin{align}
\int_{[0,1]}x^k\rho^X_{\infty}(dx) = \E_{e_k}[ U_\infty]. \label{expressionmoments}
\end{align}
Let $f \in \mathcal{C}([0,1])$ and $f_n$ be its Bernstein polynomial approximation. Letting $t\to\infty$ in \eqref{berndualfn} yields 
\begin{align}
\int_{[0,1]} f_n(z) \rho^X_{\infty}(dz) =\mathbb{E}_{v_{f,n}}\left[U_{\infty}\right]. \label{berndualfninfty}
\end{align}
Note that $H(x,V_t) = U_{\infty} \mathds{1}_{\{\tau \leq t\}}+H(x,V_t) \mathds{1}_{\{\tau > t\}}$. Combining this with \eqref{berndualfn} and \eqref{berndualfninfty} we get 
\begin{align}
\left | \mathbb{E}_x[f_n(X_t)] - \int_{[0,1]} f_n(z) \rho^X_{\infty}(dz) \right | \leq \mathbb{E}_{v_{f,n}}\left[|U_{\infty}-H(x,V_t)|\mathds{1}_{\{\tau > t\}}\right]. \label{eqlemboundespfx1}
\end{align}
Thanks to \eqref{normineq}, we have $|H(x,V_t)|\leq ||v_{f,n}||_{\infty}$ and $|U_{\infty}|\leq ||v_{f,n}||_{\infty}$. Since, in addition, we have $||v_{f,n}||_{\infty} = \max_{i \in \{0,...,n\}} |f(i/n)| \leq ||f||_{\infty}$, we infer that
\begin{align*}
\left | \mathbb{E}_x[f_n(X_t)] - \int_{[0,1]} f_n(z) \rho^X_{\infty}(dz) \right | \leq 2 ||f||_{\infty} \mathbb{P}_{n}(\tau > t) = 2 ||f||_{\infty} \mathbb{P}_n(L_t \neq 0). 
\end{align*}
Writing $f_n=f+(f_n-f)$ yields that, for all $n\in\Nb$, 
\begin{align}
\left | \mathbb{E}_x[f(X_t)] - \int_{[0,1]} f(z) \rho^X_{\infty}(dz) \right | \leq 2 || f - f_n ||_{\infty} + 2 ||f||_{\infty} \mathbb{P}_n(L_t \neq 0). \label{finalest}
\end{align}
Combining this with \eqref{bernunifcvth} we obtain
\begin{align}
\left | \mathbb{E}_x[f(X_t)] - \int_{[0,1]} f(z) \rho^X_{\infty}(dz) \right | \leq 2 ||f||_{\infty}  \sup_{n \geq 1} \mathbb{P}_n(L_t \neq 0), \label{eqlemboundespfxcdi}
\end{align}
and \eqref{eqboundradon} then follows from \eqref{eqlemboundespfxcdi} together with \eqref{defradondist} and \eqref{inegradonwass}.

Now, let $f \in \mathcal{L}_1([0,1])$.  Set $\tilde f(\cdot)\coloneqq f(\cdot)-f(1/2)$ and note that the left-hand side of \eqref{finalest} is the same for $f$ and $\tilde f$ and that $|| \tilde f - \tilde f_n ||_{\infty}=|| f - f_n ||_{\infty}$ and $2 ||\tilde f||_{\infty} \leq 1$. 
Thus, applying \eqref{finalest} to $\tilde f$ in combination with \eqref{bernunifcvthlip}, we obtain 
\begin{align}
\left | \mathbb{E}_x[f(X_t)] - \int_{[0,1]} f(z) \rho^X_{\infty}(dz) \right | \leq \frac{1}{\sqrt{n}} + \mathbb{P}_n(L_t \neq 0). \label{eqlemboundespfxgen}
\end{align}
Therefore, \eqref{eqboundwass} follows from the combination of \eqref{eqlemboundespfxgen} with \eqref{kantwassdist}.

Finally, \eqref{eqboundwass} and the assumption that $L$ is absorbed at $0$ almost surely imply that $X$ is $W_1$-uniformly ergodic on $[0,1]$. 
\end{proof}
\subsection{Distance to stationarity for $\Lambda$-Wright--Fischer processes with mutation}\label{probextasg}
This section is devoted to the proofs of Theorem \ref{thbounddiststa} and Proposition \ref{propcond}. The proof of Theorem \ref{thbounddiststa} will follow from Theorem \ref{propbounddtv} and appropriate bounds for the non-extinction probability $\mathbb{P}_n(L_t \neq 0)$ that we will obtain here.

For our analysis, it is the dynamics of the process $L$ that will be relevant, as opposed to the dynamics of the Bernstein dual $V$, which are not required in this section. Throughout this section, we will assume that the assumption \eqref{exporate} on the parameter $C$ is satisfied, i.e. $C>0$. Recall from \cite{cordhumvech2024} that the rate matrix $\Bc^L \in \R^{\Nb_0 \times \Nb_0}$ of the process $L$ is given by 
\begin{align*}
\Bc^L f(n) & = \sum_{k=2}^n \binom{n}{k} \lambda_{n,k} (f(n-k+1)-f(n)) \\
&+ \sum_{k=1}^n \binom{n}{k}\sigma_{n,k} (f(n+k)-f(n))  + \sum_{\ell=2}^{\kappa} n\beta_\ell (f(n+\ell-1)-f(n)) \\
&+ \sum_{k=1}^n \binom{n}{k}m_{n,k} (f(n-k)-f(n)) + n (\theta_a+\theta_A) (f(n-1)-f(n)),
\end{align*}
where
		\begin{equation*}
			\lambda_{n,k}\coloneqq \int_{[0,1]} z^{k-2}(1-z)^{n-k}\Lambda(\dd z), 
		\end{equation*}
		with the convention that $0^0=1$ (in particular $\lambda_{n,2}=\Lambda(\{0\})+\int_{(0,1]} (1-z)^{n-2}\Lambda(\dd z)$), $m_{n,\ell}\coloneqq m^a_{n,\ell}+m^A_{n,\ell}$ with
		\begin{equation*}
			m^a_{n,\ell}=\int_{(0,1)} z^{\ell-1}(1-z)^{n-\ell}\nu( \dd z)\quad \textrm{and}\quad m^A_{n,\ell}=\int_{(-1,0)} \lvert z\rvert^{\ell-1}(1-\lvert z\rvert )^{n-\ell}\nu( \dd z),
		\end{equation*}
and $\sigma_{n,\ell}=\sigma^a_{n,\ell}+\sigma^A_{n,\ell}$ with
		\begin{equation*}
			\sigma^a_{n,\ell}=\int_{(0,1)} z^{\ell-1}(1-z)^{n-\ell}\mu( \dd z)\quad \textrm{and}\quad \sigma^A_{n,\ell}=\int_{(-1,0)} \lvert z\rvert^{\ell-1}(1-\lvert z\rvert)^{n-\ell}\mu(\dd z).
		\end{equation*}
For any $i,j \in \Nb_0$ with $i \neq j$, ${\Bc}^L(i,j)$ is the total transition rate from $i$ to $j$ and, for any $i \in \Nb_0$, ${\Bc}^L(i,i):=-\sum_{j \in \Nb_0 \setminus \{i\}} {\Bc}^L(i,j)$. We see that ${\Bc}^L(i,\cdot)$ has finite support for all $i \in \Nb_0$. For any $g:\Nb_0\to\Rb$ we define ${\Bc}^Lg$ via ${\Bc}^Lg(i)=\sum_{j \in \Nb_0} {\Bc}^L(i,j) (g(j)-g(i))$. Whenever $L$ is a Feller process (which is not always the case in our setting), the operator $\Bc^L$ is the generator of the process $L$. We start with the following useful bounds for the operator $\Bc^L$. 
\begin{lemme} \label{ineqgenerat}Let $g:\Nb_0\to\Nb_0$ be the identity function, i.e. $g(n)=n$. Then, $$\Bc^L g(n) \leq -C n,\quad n\in\Nb_0,$$ with $C$ as in \eqref{exporate}. If, in addition, \eqref{thecondition} holds, then there is $\gamma\in(1,2)$ and $n_0 \in\Nb$ such that $${\Bc}^Lg(n) \leq -n^{\gamma},\quad n\geq n_0.$$
\end{lemme}

\begin{proof}
Applying the operator $\Bc^L$ to the function $g$ we get 
\begin{align*}\label{coalplusrest}
\Bc^L g(n) & = \sum_{k=2}^n \binom{n}{k} \lambda_{n,k} (n-k+1-n) + \sum_{k=1}^n \binom{n}{k}\sigma_{n,k} (n+k-n) \nonumber\\
& + \sum_{\ell=2}^{\kappa} n\beta_\ell (n+\ell-1-n) + \sum_{k=1}^n \binom{n}{k}m_{n,k} (n-k-n) + n (\theta_a+\theta_A) (n-1-n)  \nonumber\\
& = - \binom{n}{2} \Lambda(\{0\}) - \int_{(0,1]} \left ( nz - 1 + (1-z)^n \right ) \frac{\Lambda(\dd z)}{z^2}  \nonumber\\
& + n \left ( \int_{(-1,1)} |r| \,\frac{\mu(\dd r)}{|r|} + \sum_{\ell=2}^{\kappa} \beta_\ell (\ell-1) - \int_{(-1,1)} |r| \, \frac{\nu(\dd r)}{|r|} - \theta_a - \theta_A \right )  \nonumber\\
& = - \binom{n}{2} \Lambda(\{0\}) -\int_{(0,1]} \left ( nz - 1 + (1-z)^n \right ) \frac{ \Lambda(\dd z)}{z^2} - C n.
\end{align*}
Thus,
\begin{equation}\label{coalplusrest}
\Bc^L g(n)=\Bc_\Lambda^L(n)-C n
\end{equation}
with
\begin{align*}
{\Bc}^L_{\Lambda} g(n) & \coloneqq -\binom{n}{2} \Lambda(\{0\}) - \int_{(0,1]} \left ( nz - 1 + (1-z)^n \right )  \frac{\Lambda(\dd z)}{z^2}.
\end{align*}
Since ${\Bc}^L_{\Lambda} g(n)\leq 0$, the first bound follows. For the second bound, note that
\begin{align*}
{\Bc}^L_{\Lambda} g(n) & \leq -\binom{n}{2} \Lambda(\{0\}) - \int_{(0,1/n)} \left ( \sum_{k=2}^n \binom{n}{k} (-1)^k z^{k} \right ) \frac{\Lambda(\dd z)}{z^2}. 
\end{align*}
For $z \in (0,1/n)$, the sum in the previous expression is an alternating sum, whose terms have a decreasing absolute value. Thus, we have $\sum_{k=3}^n \binom{n}{k} (-1)^k z^{k} \geq -\binom{n}{3} z^{3}$ and therefore 
\begin{align}
{\Bc}^L_{\Lambda} g(n) & \leq -\binom{n}{2} \Lambda(\{0\}) - \int_{(0,1/n)} \left ( \binom{n}{2} z^{2} - \binom{n}{3} z^{3} \right ) \frac{\Lambda(\dd z)}{z^2} \nonumber \\
& \leq -\binom{n}{2} \Lambda(\{0\}) - \binom{n}{2} \int_{(0,1/n)} \left ( 1 - \frac{n-2}{3} z \right ) \Lambda(\dd z) \nonumber \\
& \leq -\binom{n}{2} \Lambda(\{0\}) - \frac{2}{3} \binom{n}{2} \int_{(0,1/n)}  \Lambda(\dd z) \leq - \frac{2}{3} \binom{n}{2} \Lambda([0,1/n)). \label{boundbinlambda}
\end{align}
For any $\delta$ that lies strictly between the limsup from \eqref{thecondition} and $1$, we have $\Lambda([0,x))>x^{\delta}$ for all $x$ sufficiently small. Combining with \eqref{boundbinlambda} we get, for $n$ sufficiently large, that 
\begin{align}
{\Bc}^L_{\Lambda} g(n) \leq - \frac{2}{3} \binom{n}{2} n^{-\delta} \underset{n \rightarrow \infty}{\sim} - \frac{n^{2-\delta}}{3}. \label{boundbinlambdabis}
\end{align}
Using this together with \eqref{coalplusrest} yields the second bound for any choice of $\gamma \in (1,2-\delta)$. 
\end{proof}

\begin{lemme} \label{bornespltlemme}
For any $n \geq 1$ and $t \geq 0$ we have 
\begin{align}
\mathbb{P}_n(L_t \neq 0) \leq \mathbb{E}_n[L_t] \leq n e^{-Ct}. \label{bornesplt}
\end{align}
\end{lemme}

\begin{proof} 
Since $\Pb_n(L_t\neq 0)=\Pb_n(L_t\geq 1)$, the first inequality in \eqref{bornesplt} follows from Markov inequality. To prove the second inequality, we would like to use Lemma \ref{ineqgenerat} and Dynkin's formula, however, since $L$ is not a Feller process in general, this requires some carefulness. Let us fix $a_0$ and set $T_{a_0}:=\inf \{t\geq 0, L_t>a_0\}$ (note that $T_{a_0}=0$ if $L_0>a_0$). Let $L^{a_0}$ be the process $L$ sent to the absorbing state $0$ upon leaving $[a_0]$; more precisely, $L^{a_0}_t:=L_t \mathds{1}_{t<T_{a_0}}$. Since $L^{a_0}$ is a Markov process on a finite set it is a Feller process and the function $g$ from Lemma \ref{ineqgenerat} (restricted to $[a_0]_0$) is in the domain of its generator $\Bc^{L,a_0}$. It is straightforward to see that $\Bc^{L,a_0}g(n)\leq \Bc^{L}g(n)$ for all $n \in [a_0]_0$. We thus get from Lemma \ref{ineqgenerat} that, for all $n \in [a_0]_0$, $\Bc^{L,a_0}g(n)\leq -C n$. We thus get from Dynkin's formula that, for $n \in [a_0]_0$ and $t\geq 0$, 
\begin{align*}
\mathbb{E}_n[g(L^{a_0}_t)] = g(n) + \int_0^t \mathbb{E}_n[\Bc^{L,a_0} g(L^{a_0}_s)] \dd s \leq g(n) - C \int_0^t \mathbb{E}_n[g(L^{a_0}_s)] \dd s
.\end{align*}
Thus, we have $\frac{\dd}{\dd t} \mathbb{E}_n[g(L^{a_0}_t)] \leq - C \mathbb{E}_n[g(L^{a_0}_t)]$, and hence, $\mathbb{E}_n[L_t \mathds{1}_{t<T_{a_0}}]=\mathbb{E}_n[g(L^{a_0}_t)] \leq g(n) e^{-Ct}=n e^{-Ct}$. Now fixing $t>0$ we see that, since $L$ is conservative by \cite[Lemma 4.1]{cordhumvech2024}, we have almost surely that $T_{a_0}>t$ for all large $a_0$ so the second inequality in \eqref{bornesplt} follows by monotone convergence. 
\end{proof}

\begin{lemme} \label{bornespltcasecdi}
Assume that $\sup_{n \geq 1} \mathbb{E}_n[L_1] < \infty$ and set $K\coloneqq e^C \sup_{n \geq 1} \mathbb{E}_n[L_1]$. Then, we have, for any $n \geq 1$ and $t \geq 1$,  
\begin{align}
\mathbb{P}_n(L_t \neq 0) \leq \mathbb{E}_n[L_t] \leq K e^{-Ct}. \label{borneunif}
\end{align}
\end{lemme}

\begin{proof} 
The first inequality was already proved in Lemma \ref{bornespltlemme}. It remains to prove the second inequality in \eqref{borneunif}. Using the Markov property at time $1$ for the process $L$ and Lemma \ref{bornespltlemme}, we get that, for any $n \geq 1$ and $t \geq 1$, 
\begin{align*}
\mathbb{E}_n[L_t] = \mathbb{E}_n[\mathbb{E}_n[L_t|\mathcal{F}_1]] \leq \mathbb{E}_n[L_1] e^{-C(t-1)} \leq K e^{-Ct},
\end{align*}
which concludes the proof.
\end{proof}
\begin{proof}[Proof of Theorem \ref{thbounddiststa}]
The combination of \eqref{eqboundradon} from Theorem \ref{propbounddtv} with Lemma \ref{bornespltcasecdi} directly yields \eqref{finalboundradon}. 
Combining \eqref{eqboundwass} from Theorem \ref{propbounddtv} with Lemma \ref{bornespltlemme} we get that, for $x \in [0,1]$, $n\geq 1$ and $t \geq 0$ we have 
\begin{align}
W_1 (\rho_{x,t}^X, \rho^X_{\infty}) \leq \frac{1}{\sqrt{n}} + n e^{-Ct}. \label{majon}
\end{align}
Using this inequality with  $n\coloneqq\lfloor 2 e^{2Ct/3} \rfloor$, and noticing that $e^{2Ct/3} \leq n \leq 2e^{2Ct/3}$, we get \eqref{finalboundwass}. The $W_1$-uniform ergodicity of $X$ on $[0,1]$ follows. 

Now, let us assume that $\Lambda([0,1/2+\epsilon])>0$ for every $\epsilon>0$ and show that $X$ is open set recurrent. According to Corollary \ref{m2smallerthanm1} we have $\E_{e_2}[ U_\infty]<\E_{e_1}[ U_\infty]$, which combined with \eqref{expressionmoments} yields $$\int_{[0,1]}x^2\rho^X_{\infty}(\dd x)<\int_{[0,1]}x\rho^X_{\infty}(\dd x).$$ It follows that $\rho^X_{\infty}((0,1))>0$. 
%\begin{align}
%\int_{[0,1]}x^2\rho^X_{\infty}(dx)<\int_{[0,1]}x\rho^X_{\infty}(dx). \label{notabernoulli}
%\end{align}
We infer that $(0,1)$ contains a point of $Supp \ \rho^X_{\infty}$ (see section \ref{notations}). Proposition \ref{suppofmuinftyisrecunif} then implies that $\mathcal{R}(X) \cap (0,1) \neq \emptyset$. Hence, according to Theorem \ref{fullsupport}, the process $X$ is open set recurrent.
\end{proof}
The remainder of this section is devoted to the proof of Proposition \ref{propcond}, which states that \eqref{thecondition} implies the generic condition $\sup_{n}\Eb_n[L_1]<\infty$. 
%%%%%%%%%%%%%%%%%%%%%%%%%%%%%%%%%%%%%%%%%

\begin{proof}[Proof of Proposition \ref{propcond}]
Let $n_0 \geq 1$ and $\gamma\in(1,2)$ be as in Lemma \ref{ineqgenerat}. Let $K\coloneqq K(n_0,\gamma)>n_0$ to be appropriately chosen later. Note from Lemma \ref{bornespltlemme} that
\begin{equation*}
\sup_{n\leq K}\Eb_n[L_1]\leq c(K)<\infty,
\end{equation*}
where the constant $c(K)>0$ only depends on $K$. It remains to show that $\sup_{n >K}\Eb_n[L_1]\leq \tilde{c}(K)<\infty$ for some constant $\tilde{c}(K)>0$.
Let us fix $a_0>K$ and set $L^{a_0}$ and $\Bc^{L,a_0}$ as in the proof of Lemma \ref{bornespltlemme}. We define $\phi^{a_0}_n(t)\coloneqq \Eb_n[L^{a_0}_t]$, $t\geq 0$. We thus get from Dynkin's formula that, for $n \in [a_0]_0$ and $t\geq 0$, 
\begin{align*}
\frac{\dd}{\dd t} \phi^{a_0}_n(t)&=\Eb_n[\Bc^{L,a_0} g(L^{a_0}_t)]=\Eb_n[\Bc^{L,a_0} g(L^{a_0}_t) 1_{\{L^{a_0}_t \geq n_0\}}]+\Eb_n[\Bc^{L,a_0} g(L^{a_0}_t) 1_{\{L^{a_0}_t<n_0\}}].\end{align*}
We recall that $\Bc^{L,a_0}g(n)\leq \Bc^{L}g(n)$ for all $n \in [a_0]_0$. We thus get from Lemma \ref{ineqgenerat} that, for all integers $n \in [n_0,a_0]$, $\Bc^{L,a_0}g(n)\leq -n^{\gamma}$ and for all integers $n \in [n_0]_0$, $\Bc^{L,a_0}g(n)\leq -Cn$. Using this and Jensen's inequality, we obtain
\begin{equation*}
\frac{\dd}{\dd t} \phi^{a_0}_n(t)\leq -\phi^{a_0}_n(t)^{\gamma}+ \Eb_n[(-C L^{a_0}_t+(L^{a_0}_t)^\gamma)1_{\{L^{a_0}_t<n_0\}}],\qquad n \in [K,a_0].
\end{equation*}
Setting $c_1(n_0)\coloneqq |C|n_0 + n_0^{\gamma}$ we get
\begin{equation}\label{inteq2}
\frac{\dd}{\dd t} \phi^{a_0}_n(t) \leq -(\phi^{a_0}_n(t))^{\gamma}+ c_1(n_0),\qquad t\in[0,1], n \in [K,a_0].
\end{equation}
Let $\delta=(1+\gamma)/2\in(1,\gamma)$ and set $$K\coloneqq K(n_0,\gamma)\coloneqq\min\{k\geq n_0: k^{\gamma}-k^{\delta}>c_1(n_0)\}.$$
Note that if $x\geq K$, then $x^{\gamma}-x^{\delta}\geq K^{\gamma}-K^{\delta}>c_1(n_0)$. 
Let now $n> K$. Since $\phi^{a_0}_n(0)=n> K$, we deduce that $$t_{K,a_0}(n)\coloneqq\inf\{s\geq 0:\phi^{a_0}_n(s)\leq K\}\wedge 1>0.$$
Moreover, by construction we have, for $t\in[0,t_{K,a_0}(n))$ 
\begin{equation*}
\frac{\dd}{\dd t} \phi^{a_0}_n(t)\leq -(\phi^{a_0}_n(t))^{\delta}.
\end{equation*}
This yields
$$\phi^{a_0}_n(t)\leq {\left(\frac{1}{n^{\delta-1}}+(\delta-1)t\right)^{-\frac{1}{\delta-1}}},\qquad t\in[0,t_{K,a_0}(n)).$$
Now, we have two options, either $t_{K,a_0}(n)\geq 1$, in which case $\phi^{a_0}_n(1)\leq (\delta-1)^{-\frac{1}{\delta-1}}$, or $t_{K,a_0}(n)<1$, in which case $\phi^{a_0}_n(1)\leq K$ (because $\frac{\dd}{\dd t} \phi^{a_0}_n(\cdot)$ is negative on $(\phi^{a_0}_n)^{-1}(\{K\})$). We conclude that
$$\sup_{n \in [K,a_0]}\phi^{a_0}_n(1)\leq K\vee  (\delta-1)^{-\frac{1}{\delta-1}}.$$
Letting $a_0$ go to infinity the desired result follows. 
\end{proof}
\section{Proofs of results for processes with Siegmund dual} \label{siegmundsection}
This section is devoted to proving our main results for processes with a Siegmund dual. More precisely, we will prove Theorem \ref{genericsieg} and its applications to $\Lambda$-Wright--Fisher models without mutation given by Theorems \ref{speedexpoofcvtostationarity} and \ref{speedpolofcvtostationarity}.
\subsection{Distance to stationarity via Siegmund duality}
We start with the proof of Theorem \ref{genericsieg} about the distance to stationarity for $[0,1]$-valued processes with a Siegmund dual.

\begin{proof}[Proof of Theorem \ref{genericsieg}]
Recall that the Siegmund dual $Y$  of $X$ is assumed to be asymptotically absorbed in $\{0,1\}$. Therefore, for any $y \in [0,1]$, we have $\P_y$-almost surely that $Y_{\infty}\coloneqq \lim_{t\to\infty}Y_t$ exists and belongs to $\{0,1\}$. Thus, for all $x\in (0,1)$, 
\begin{align} \label{cvfctrep}
\P_x(X_t \geq y)=\P_y(x \geq Y_t)\xrightarrow[t\to\infty]{} \P_y(Y_{ \infty}=0).
\end{align}
In particular, the distribution function of $X_t$ under $\P_x$ converges as $t\to\infty$. Since $X$ takes values on $[0,1]$, we infer that $(X_t)_{t>0}$ is a tight family of random variables. Thus, there is a distribution $\rho^X_{\infty}$ on $[0,1]$ such that, for any $x \in (0,1)$, the law of $X_t$ under $\P_x$ converges weakly to $\rho^X_{\infty}$. Moreover we get from \eqref{cvfctrep} that, for any $y \in [0,1]$, we have 
\begin{align}
\rho^X_{\infty}([y,1])=\mathbb{P}_{y} (Y_{\infty}=0), \qquad \rho^X_{\infty}([0,y))=\mathbb{P}_y (Y_{\infty}=1). \label{stameasxtoy}
\end{align}
Combining this with the assumption that $Y$ is regularly absorbed in $\{0,1\}$, we get that $\rho^X_{\infty}(\{0,1\})=0$. 

Since $X$ satisfies Condition \ref{moncoup}, for any $x_1, x_2 \in [0,1]$, there is a $(x_1, x_2)$-monotone coupling $(X^1_t, X^2_t)_{t \geq 0}$ and we denote its law by $\mathbb{P}_{x_1,x_2}$. Let us fix $a\in (0,1/2)$ and $x \in [a,1-a]$. Thanks to the measurability requirement in Condition \ref{moncoup}, we can define the distribution $\mathbb{P}_{x,\rho^X_{\infty}}(\cdot):=\int_{[0,1]} \mathbb{P}_{x,z}(\cdot) \rho^X_{\infty}(\dd z)$. Since $\int_{[0,1]} \mathbb{P}_{z}(X_t \in \cdot) \rho^X_{\infty}(\dd z)=\rho^X_{\infty}(\cdot)$, the process $(X^1_t, X^2_t)$ with law $\mathbb{P}_{x,\rho^X_{\infty}}(\cdot)$ is a coupling with marginals $\rho^X_{x,t}$ and $\rho^X_{\infty}$. We conclude from \eqref{defwassdist} that 
\begin{align}
W_p(\rho^X_{x,t},\rho^X_{\infty})^p & \leq \int_{[0,1]} \mathbb{E}_{x,z} \left [ |X^2_t - X^1_t|^p \right ] \rho^X_{\infty}(\dd z) \nonumber \\
&\leq \int_{[0,1]} \mathbb{E}_{x\wedge z,x \vee z} \left [ (X^2_t)^p - (X^1_t)^p \right ] \rho^X_{\infty}(\dd z) \nonumber \\
& = \int_{[0,1]} \mathbb{E}_{x\wedge z,x \vee z} \left [ \int_{[0,1]} p u^{p-1} \mathds{1}_{\{X^1_t<u\leq X^2_t\}} \right ] \rho^X_{\infty}(\dd z) \nonumber \\
& = \int_{[0,1]} \int_{[0,1]} p u^{p-1} \mathbb{P}_{x\wedge z,x \vee z} \left ( X^1_t<u\leq X^2_t \right ) \rho^X_{\infty}(\dd z) \nonumber \\
& = \int_{[0,1]} \int_{[0,1]} p u^{p-1} \mathbb{P}_{u} \left ( x\wedge z <Y_t\leq x \vee z \right ) \rho^X_{\infty}(\dd z) \nonumber \\
& \leq \sup_{y \in [0,1]} \mathbb{P}_y \left ( Y_t\in [a,1-a] \right ) + \rho^X_{\infty}([0,a))+\rho^X_{\infty}([1-a,1]), \label{2trajx<absynew}
\end{align}
where, in the second-to-last line of the calculation, we used \eqref{doubleduality}. The combination of \eqref{2trajx<absynew} with \eqref{stameasxtoy} yields \eqref{decaywpsiegmund}. Then, the combination of \eqref{decaywpsiegmund} with the assumption that $Y$ is regularly and uniformly absorbed in $\{0,1\}$ implies that, for all $p\geq 1$, $X$ is locally $W_p$ uniformly ergodic in $(0,1)$. 
\end{proof}

\subsection{Distance to stationarity for $\Lambda$-Wright--Fischer processes without mutation in the coexistence regime}\label{estimsiegdual} Now we move to the applications of the Siegmund duality method to $\Lambda$-Wright--Fisher processes without mutation. More precisely, in this section we will provide the proofs of Theorems \ref{speedexpoofcvtostationarity} and \ref{speedpolofcvtostationarity}. 
 
Recall that here we assume that the parameters of the model satisfy \eqref{assumptionnomut}. In this setting, it was proved in \cite[Thm. 2.5]{cordhumvech2022} that $X$ is Siegmund dual to the solution of the SDE
\begin{align}\label{eq:SDEWFPdual}
\dd Y_t \  & = \int\limits_{(0,1)\times [0,1]}m(Y_{t-},r,u){N}(\dd t, \dd r,\dd u) + \int\limits_{(-1,1)} g(Y_{t-},r) S(\dd t, \dd r) -\sigma(Y_{t})Y_t(1-Y_t) \dd t,
\end{align}
where 
\begin{align*}
m(y,r,u) \coloneqq \mathsf{Median} \left \{ \frac{y-r}{1-r}, \frac{y}{1-r}, u \right \}-y, \ g(y,r) \coloneqq \frac{1+r - \sqrt{(1+r)^2 - 4ry}}{2r} - y, 
\end{align*}
Moreover, one can easily check using Theorem 2.3 in \cite{DaLi12} that the process $X$ satisfies Condition \ref{moncoup} (see Prop. A.3 in \cite{cordhumvech2022} for the details). Hence, we are in the framework of Theorem \ref{genericsieg}. Therefore, estimating the distance $W_p(\rho^X_{x,t}, \rho^X_{\infty})$ via the Siegmund dual process $Y$ involves two key steps. First, one must estimate the probability $\mathbb{P}_y(Y_t \in [a, 1 - a])$, which quantifies the likelihood that the dual process remains away from both boundaries at time $t$. Second, one must estimate the absorption probabilities of $Y$ when it starts near the boundaries. The first of these steps is addressed in the following result, originally proved in \cite{cordhumvech2022}.

\begin{prop} [\cite{cordhumvech2022}, Theorem 3.3]\label{thm:coexistenceY} Assume that \eqref{assumptionnomut} and \eqref{assumptionc0c1} hold. For any $\ell\in (0,C_0(\Lambda,\mu,\sigma) \wedge C_1(\Lambda,\mu,\sigma))$ and $y\in (0,1)$ we have $\mathbb{P}_y$-almost surely that either a) $Y_t \in [0,e^{-\ell t}]$ for all large $t$, or 
 b) $Y_t \in [1-e^{-\ell t},1]$ for all large $t$. Moreover, 
	
	\begin{enumerate}
		\item[(i)] (polynomial decay)  If \eqref{weakhyp} holds for some $\gamma>0$, then for any {$\ell\in (0,C_0(\Lambda,\mu,\sigma) \wedge C_1(\Lambda,\mu,\sigma))$ }and $\alpha\in (0,\gamma)$ there is $K=K(\alpha,\ell)$ such that for any $t\geq 0$ we have 
        \begin{align}
			\sup_{y \in [0,1]} \P_y(Y_t\in [e^{-\ell t},1-e^{-\ell t}])\leq Kt^{-\alpha}. \label{yinpoldecay}
		\end{align}
		\item[(ii)] (exponential decay) If \eqref{stronghyp} holds for some $\gamma>0$, then for any $\ell\in (0,C_0(\Lambda,\mu,\sigma) \wedge C_1(\Lambda,\mu,\sigma))$ there are positive constants $K_1=K_1(\ell), K_2=K_2(\ell)$ such that for any $t\geq 0$ we have \begin{align}
			\sup_{y \in [0,1]} \P_y(Y_t\in [e^{-\ell t},1-e^{-\ell t}])\leq K_1e^{-K_2t}. \label{yinexpodecay}
		\end{align}
		
	\end{enumerate}
\end{prop}

For the second step, we now prove the following proposition that provides estimates for the absorption probabilities of the Siegmund dual $Y$. 
\begin{prop} \label{queueexpopix}
Assume that \eqref{assumptionnomut} and \eqref{assumptionc0c1} are satisfied. 
\begin{enumerate}
\item[(i)] (weak regularity) If \eqref{weakhyp} holds for some $\gamma>0$, then for any $\alpha \in (0,(1+\gamma)\gamma/(2\gamma+1))$ there is a positive constant $K=K(\alpha)$ such that, for any $y \in (0,1/2)$, we have 
\begin{align}
\mathbb{P}_y (Y_{\infty}=1)+\mathbb{P}_{1-y} (Y_{\infty}=0) \leq K (\log(1/y))^{-\alpha}. \label{tailpisoushypmomentspol}
\end{align}
\item[(ii)] (strong regularity) If \eqref{stronghyp} holds for some $\gamma>0$, then there are positive constants $K$ and $\alpha$ such that, for any $y \in (0,1/2)$, we have 
\begin{align}
\mathbb{P}_y (Y_{\infty}=1)+\mathbb{P}_{1-y} (Y_{\infty}=0) \leq K y^{\alpha}. \label{tailpisoushypmomentsexpo}
\end{align}
\end{enumerate}
\end{prop}

In order to prove Proposition \ref{queueexpopix} we use the comparison results from \cite{cordhumvech2022} between $Y$ and functions of L\'evy processes near the boundaries. For $b\geq \log(2)$, define the L\'evy processes $L^b=(L_t^b)_{t\geq 0}$ via
\begin{align}
L_t^b & \coloneqq \!\int\limits_{[0,t]\times(0,1)^2}\!\!\!\!\!\!\!\!\!\!\log(1-r)N(\dd s,\dd r,\dd u) + \frac{1}{2}\int\limits_{[0,t]\times(-1,1)}\!\!\!\!\!\!\!\!\!\!\log( (1+r)^2-4re^{-b}\mathds{1}_{r>0} ) S(\dd s, \dd r)\nonumber \\
	& \qquad +t\big(\sigma(0)-e^{-b}\lVert \sigma\rVert_{\mathcal{C}^1([0,1])}\big)\label{eq:levylower}. 
\end{align}
The L\'evy process $L^b$ is well-defined under the assumption \eqref{assumptionnomut}, see \cite[Lemma 4.1]
{cordhumvech2022}. For a measurable set $A\subset[0,1]$, we set $T_YA\coloneqq\inf\{s\geq 0: Y_s\in A\}$. The comparison result can now be stated as follows. 
\begin{lemme}[\cite{cordhumvech2022}, Lemma 4.2]\label{lem:levysandwich}
	Assume \eqref{assumptionnomut} and fix $b\geq \log(2)$ and $Y_0\in (0,e^{-b}]$. 
	Almost surely, for all $t\in [0,T_{Y}(e^{-b},1]]$, \begin{equation}
		Y_t \leq Y_0 e^{-L_t^b}.\label{eq:comparison}
	\end{equation}
\end{lemme}

We also need some properties of the L\'evy processes $L^b$ that are gathered in the following lemma, proved in \cite{cordhumvech2022}. 
\begin{lemme}\label{lem:existence_suit_approx}
	Assume that \eqref{assumptionnomut} holds and that  $C_0(\Lambda,\mu,\sigma)>0$. Then there exists $\beta \geq \log(2)$ such that for all $b\geq \beta$, $\E[L_1^b]$ is well-defined and $\E[L_1^b]>0$. In particular, $L^{b}$ drifts to $\infty$ so $\inf_{s \geq 0}L_s^b>-\infty$. If moreover we assume that \eqref{stronghyp} is satisfied for some $\gamma>0$ then $\E[e^{\lambda L_1^{b}}]<\infty$ for all $\lambda\in [-\gamma,\infty)$. 
\end{lemme}

\begin{proof}
The first and second part follows from respectively \cite[Lemma 4.8]{cordhumvech2022} and \cite[Lemma 4.4]{cordhumvech2022}. 
\end{proof}

We now have all the ingredients to prove Proposition \ref{queueexpopix}. 

\begin{proof} [Proof of Proposition \ref{queueexpopix}]
Let us fix $b$ as in Lemma \ref{lem:existence_suit_approx}. For $y \in (0,e^{-b})$, using Lemma \ref{lem:levysandwich} we get 
\begin{align}
\mathbb{P}_y(Y_{\infty}=1) \leq \mathbb{P}_y \left (\sup_{s \geq 0} Y_s > e^{-b} \right ) \leq \mathbb{P} \left (\sup_{s \geq 0} y e^{-L_s^b} > e^{-b} \right ) = \mathbb{P} \left (\inf_{s \geq 0} L^{b}_s < \log e^b y \right ). \label{relationpixlbeta}
\end{align}
By the choice of $b$ and Lemma \ref{lem:existence_suit_approx}, the random variable $\inf_{s \geq 0} \tilde L^{b,\eta}_s$ is well-defined and finite. 

Let us now assume that \eqref{stronghyp} is satisfied for some $\gamma>0$ and prove \eqref{tailpisoushypmomentsexpo}. By Lemma \ref{lem:existence_suit_approx} the Laplace transform of $L^{b}_1$ is well-defined on $[-\gamma,\infty)$. We define the Laplace exponent $\psi_{b}(\cdot)$ by $\psi_{b}(\lambda)\coloneqq\log(\E[e^{\lambda L_1^{b}}])$ and note from the lemma that $\psi_{b}'(0)=\mathbb{E}[L^{b}_1]>0$. As a consequence, for all $\lambda \in (0,\gamma)$ small enough we have $\psi_{b}(-\lambda) < 0$. We thus get that $-L^{b}$ satisfies the requirements of Lemma \ref{expomomentsup} so $\mathbb{E}[e^{-c \inf_{s \geq 0} L^{b}_s}]<\infty$ for some $c>0$. Combining this with \eqref{relationpixlbeta} and Chernoff inequality we get $\mathbb{P}_y(Y_{\infty}=1) \leq \mathbb{E}[e^{-c \inf_{s \geq 0} \tilde L^{b,\eta}_s}] e^{cb} y^c$. We deduce the existence of constants $c_1, \alpha_1>0$ such that for any $y \in (0,1)$, $\mathbb{P}_y(Y_{\infty}=1) \leq c_1 y^{\alpha_1}$. Now applying the same proof to the process $\hat Y \coloneqq 1-Y$ yields the existence of constants $c_2, \alpha_2 > 0$ such that for any $y \in (0,1)$, $\mathbb{P}_{1-y} (Y_{\infty}=0) \leq c_2 y^{\alpha_2}$. This entails \eqref{tailpisoushypmomentsexpo}. 

Let us now assume that \eqref{weakhyp} is satisfied for some $\gamma>0$ and prove \eqref{tailpisoushypmomentspol}. The condition \eqref{weakhyp} ensures that $-L^{b}$ satisfies the requirements of Lemma \ref{polmomentvaluesup} with $m=1+\gamma$ so $\mathbb{E}[(-\inf_{s \geq 0} L^{b}_s)^{\alpha}]<\infty$ for every $\alpha \in (0,(1+\gamma)\gamma/(2\gamma+1))$. Using \eqref{relationpixlbeta} and Markov inequality we get $\mathbb{P}_y(Y_{\infty}=1) \leq \mathbb{E}[(- \inf_{s \geq 0} \tilde L^{b,\eta}_s)^{\alpha}] (\log(e^{-b}/y))^{-\alpha}$. We deduce the existence of a constant $C_1 > 0$ such that for any $y \in (0,1)$, $\mathbb{P}_y(Y_{\infty}=1) \leq C_1 (\log(1/y))^{-\alpha}$. Now applying the same proof to the process $\hat Y \coloneqq 1-Y$ yields the existence of a constant $C_2 > 0$ such that for any $y \in (0,1)$, $\mathbb{P}_{1-y} (Y_{\infty}=0) \leq C_2 (\log(1/y))^{-\alpha}$. This entails \eqref{tailpisoushypmomentspol}. 
\end{proof}

\subsection{Proof of main result} \label{proofmainresult}

We now combine Theorem \ref{genericsieg} with the estimates from Section \ref{estimsiegdual} to prove Theorems \ref{speedexpoofcvtostationarity} and \ref{speedpolofcvtostationarity}. 

\begin{proof}[Proof of Theorem \ref{speedexpoofcvtostationarity}]

We now assume that \eqref{stronghyp} is satisfied for some $\gamma>0$. Let us fix a number $\ell\in (0,C_0(\Lambda,\mu,\sigma) \wedge C_1(\Lambda,\mu,\sigma))$. By \eqref{yinexpodecay} from Proposition \ref{thm:coexistenceY} we see that $Y$ is uniformly absorbed in $\{0,1\}$ and by \eqref{tailpisoushypmomentsexpo} from Proposition \ref{queueexpopix} we see that $Y$ is regularly absorbed in $\{0,1\}$. By Theorem \ref{genericsieg} we get that, for any $p \geq 1$, $X$ is locally $W_p$-uniformly ergodic on $(0,1)$ and that $\rho^X_{\infty}(\{0,1\})=0$ (where we have denoted by $\rho^X_{\infty}$ the stationary distribution of $X$). We thus have $\rho^X_{\infty}((0,1))=\rho^X_{\infty}([0,1])=1>0$, so we can proceed as in the proof of Theorem \ref{thbounddiststa} (but using Proposition \ref{suppofmuinftyisrec} instead of Proposition \ref{suppofmuinftyisrecunif}) to prove the claim about open set recurrence. 

We are now left to prove \eqref{boundexpowasserstein}. Let us fix $p \geq 1$ and $\epsilon \in (0,1/2)$. Applying \eqref{decaywpsiegmund} from Theorem \ref{genericsieg} with the choice $a\coloneqq e^{-\ell t}$ and combining with \eqref{yinexpodecay} from Proposition \ref{thm:coexistenceY} and \eqref{tailpisoushypmomentsexpo} from Proposition \ref{queueexpopix} we get that, for all $x \in [\epsilon,1-\epsilon]$ and $t \geq \log(1/\epsilon)/\ell$, 
\begin{align*}
W_p(\rho^X_{x,t},\rho^X_{\infty})^p & \leq \sup_{y \in [0,1]} \mathbb{P}_y \left ( Y_t\in [e^{-\ell t},1-e^{-\ell t}] \right ) + \mathbb{P}_{e^{-\ell t}} (Y_{\infty}=1) + \mathbb{P}_{1-e^{-\ell t}} (Y_{\infty}=0) \\
& \leq K_1e^{-K_2t} + K e^{-\ell \alpha t}. 
\end{align*}
Since $W_p(\rho^X_{x,t},\rho^X_{\infty})^p$ is always smaller than $1$ we have, by increasing the constant factors in the right-hand side if necessary, that \eqref{boundexpowasserstein} holds true for all $x \in [\epsilon, 1-\epsilon]$ and $t \geq 0$. 
\end{proof}

\begin{proof}[Proof of Theorem \ref{speedpolofcvtostationarity}]
Assume that \eqref{weakhyp} is satisfied for some $\gamma>0$. Proceeding as in the proof of Theorem \ref{speedexpoofcvtostationarity} we get that, for any $p \geq 1$, $X$ is locally $W_p$-uniformly ergodic on $(0,1)$ and the claim about open set recurrence. We now prove \eqref{boundpolwasserstein}. Let us fix $\ell\in (0,C_0(\Lambda,\mu,\sigma) \wedge C_1(\Lambda,\mu,\sigma))$, $p \geq 1$, $\epsilon \in (0,1/2)$ and $\beta \in (0,(1+\gamma)\gamma/(2\gamma+1)p)$ and note that $p \beta < (1+\gamma)\gamma/(2\gamma+1) < \gamma$. Applying \eqref{decaywpsiegmund} from Theorem \ref{genericsieg} with the choice $a\coloneqq e^{-\ell t}$ and combining with \eqref{yinpoldecay} from Proposition \ref{thm:coexistenceY} and \eqref{tailpisoushypmomentspol} from Proposition \ref{queueexpopix} (both applied with the choice of $\alpha=p \beta$) we get the existence of a constant $K>0$ such that, for all $x \in [\epsilon,1-\epsilon]$ and $t \geq \log(1/\epsilon)/\ell$, 
\begin{align*}
W_p(\rho^X_{x,t},\rho^X_{\infty})^p & \leq \sup_{y \in [0,1]} \mathbb{P}_y \left ( Y_t\in [e^{-\ell t},1-e^{-\ell t}] \right ) + \mathbb{P}_{e^{-\ell t}} (Y_{\infty}=1) + \mathbb{P}_{1-e^{-\ell t}} (Y_{\infty}=0)\\
&\leq Kt^{-p\beta}. 
\end{align*}
Since $W_p(\rho^X_{x,t},\rho^X_{\infty})^p$ is always smaller than $1$ we get, by increasing the constant factors in the right-hand side if necessary, that \eqref{boundpolwasserstein} holds true for all $x \in [\epsilon, 1-\epsilon]$ and $t \geq 0$. 
\end{proof}
\section{From ergodicity to open set recurrence}\label{sec:ergotoopensetrec}
This section is devoted to proving the results stated in Section \ref{sec:opensetrec} about recurrence properties of $[0, 1]$-valued ergodic Markov processes. More precisely, we will prove Propositions  \ref{suppofmuinftyisrecunif} and \ref{suppofmuinftyisrec} on topological recurrence of the support of the corresponding stationary distribution. We will also demonstrate Theorem \ref{fullsupport} on the topological recurrence of the full interval $[0,1]$ for $\Lambda$-Wright-Fisher processes under a simple criterion for the measure $\Lambda$. We conclude the section by proving the optimality of this criterion.
\subsection{Recurrence properties for ergodic processes}\label{ergotoopensetrec}
In this section we prove Propositions  \ref{suppofmuinftyisrecunif} and \ref{suppofmuinftyisrec}. We start with the proof of the former.
\begin{proof}[Proof of Proposition \ref{suppofmuinftyisrecunif}]
Let $x \in Supp \ \rho^X_{\infty}$, $z \in (0,1)$ and $\eta>0$. Since $x \in Supp \ \rho^X_{\infty}$, we have $\delta\coloneqq\rho^X_{\infty}((x-\eta/2,x+\eta/2))>0$. Since $X$ is assumed to be $d_{LP}$-uniformly ergodic on $(0,1)$, there exists $t_0>0$ such that, for any $y \in [0,1]$, we have $d_{LP}(\rho^X_{y,t_0},\rho^X_{\infty})<(\delta\wedge\eta)/2$. By definition of $d_{LP}(\cdot,\cdot)$, we get, for all $y\in[0,1]$, 
\begin{align*}
\ \mathbb{P}_y(X_{t_0} \in (x-\eta,x+\eta)) \geq \rho^X_{\infty}((x-\eta/2,x+\eta/2)) - (\delta\wedge\eta)/2 \geq \delta/2 >0. 
\end{align*}
We thus get that, for the Markov chain $(X_{nt_0})_{n \geq 0}$ on $[0,1]$, $(x-\eta,x+\eta)$ is uniformly accessible from $[0,1]$ in the sense of \cite[(4.11)]{bookMT1993}. By \cite[Thm. 9.1.3]{bookMT1993} we get 
\begin{align*}
\mathbb{P}_z( \text{$X_{nt_0} \in (x-\eta,x+\eta)$ for infinitely many $n\geq 0$} )=1. 
\end{align*}
Since this holds for all $z \in (0,1)$ and $\eta>0$, we deduce that $x \in \mathcal{R}(X)$. Therefore, $Supp \ \rho^X_{\infty} \subset \mathcal{R}(X)$. 
\end{proof}

The case of the local uniform ergodicity condition from Proposition \ref{suppofmuinftyisrecunif} is more delicate and requires a preliminary lemma. 
\begin{lemme} \label{largeintoftenvisited}
Let $X$ be a Markov process on $[0,1]$. We assume that $X$ is ergodic on $(0,1)$ and denote its stationary distribution by $\rho^X_{\infty}$. We also assume that $\rho^X_{\infty}(\{0,1\})=0$. Then for any $x \in (0,1)$ and $t_0>0$ we have 
\begin{align}
\mathbb{P}_x( \textrm{$X_{nt_0} \in [\epsilon,1-\epsilon]$ for infinitely many $n\geq 0$} ) \underset{\epsilon \rightarrow 0}{\longrightarrow} 1. \label{largeintoftenvisited0}
\end{align}
\end{lemme}

\begin{proof}
Let 
\begin{align*}
\mathcal{E}_k \coloneqq \left \{ \text{$X_{nt_0} \in [1/k,1-1/k]$ for infinitely many $n\geq 0$} \right \}. 
\end{align*}
For any $k \geq 4$ we define a function $f_k:[0,1] \rightarrow \mathbb{R}$ by 
\begin{align*}
f_k(x) \coloneqq \left\{
\begin{aligned}
& 1 \ \text{if} \ x \in [0,1/k] \cup [1-1/k,1], \\
& -kx+2 \ \text{if} \ x \in (1/k,2/k), \\
& kx+2-k \ \text{if} \ x \in (1-2/k,1-1/k), \\
& 0 \ \text{if} \ x \in [2/k,1-2/k]. \end{aligned} \right. 
\end{align*}
Since $f_k$ is continuous, the ergodicity yields 
\begin{align}
\mathbb{E}_x[f_k(X_{nt_0})] \underset{n \rightarrow \infty}{\longrightarrow} \int_{[0,1]} f_k(z) \rho^X_{\infty}(dz) \leq \rho^X_{\infty}([0,2/k] \cup [1-2/k,1]). \label{cvefkx1}
\end{align}
On $\mathcal{E}_k^c$, $f_k(X_{nt_0})$ converges almost surely to $1$ so we get by dominated convergence that 
\begin{align}
\mathbb{E}_x[f_k(X_{nt_0}) \mathds{1}_{\mathcal{E}_k^c}] \underset{n \rightarrow \infty}{\longrightarrow} \mathbb{P}_x(\mathcal{E}_k^c). \label{cvefkx2}
\end{align}
Since $f_k$ is non-negative, the combination of \eqref{cvefkx1} and \eqref{cvefkx2} yields 
\begin{align*}
\mathbb{P}_x(\mathcal{E}_k^c) \leq \rho^X_{\infty}([0,2/k] \cup [1-2/k,1]). 
\end{align*}
Then, by the assumption $\rho^X_{\infty}(\{0,1\})=0$, the right-hand side converges to $0$ as $k \rightarrow \infty$. We thus get $\mathbb{P}_x(\mathcal{E}_k^c) \rightarrow 0$ as $k \rightarrow \infty$, which yields \eqref{largeintoftenvisited0}. 
\end{proof}
Equipped with the previous lemma, we can now proceed to show Proposition \ref{suppofmuinftyisrec}.
\begin{proof}[Proof of Proposition \ref{suppofmuinftyisrec}]
Let $x \in Supp \ \rho^X_{\infty}$, $z \in (0,1)$, $\eta>0$ and $\epsilon \in (0,1/2)$. Since $x \in Supp \ \rho^X_{\infty}$ we have $\delta\coloneqq\rho^X_{\infty}((x-\eta/2,x+\eta/2))>0$. By the assumption that $X$ is locally $d_{LP}$-uniformly ergodic on $(0,1)$, there exists $t_0>0$ such that for any $y \in [\epsilon,1-\epsilon]$ we have $d_{LP}(\rho^X_{y,t_0},\rho^X_{\infty})<(\delta\wedge\eta)/2$. By definition of $d_{LP}(\cdot,\cdot)$ this yields 
\begin{align*}
\forall y \in [\epsilon,1-\epsilon], \ \mathbb{P}_y(X_{t_0} \in (x-\eta,x+\eta)) \geq \rho^X_{\infty}((x-\eta/2,x+\eta/2)) - (\delta\wedge\eta)/2 \geq \delta/2 >0. 
\end{align*}
We thus get that, for the Markov chain $(X_{nt_0})_{n \geq 0}$ on $[0,1]$, $(x-\eta,x+\eta)$ is uniformly accessible from $[\epsilon,1-\epsilon]$ in the sense of \cite[(4.11)]{bookMT1993}. By \cite[Thm. 9.1.3]{bookMT1993} we get that 
\begin{align}
& \mathbb{P}_z( \text{$X_{nt_0} \in [\epsilon,1-\epsilon]$ for infinitely many $n\geq 0$} ) \nonumber \\
& \qquad \qquad \qquad \leq \mathbb{P}_z( \text{$X_{nt_0} \in (x-\eta,x+\eta)$ for infinitely many $n\geq 0$} ). \label{probaxvisitsneighborhood}
\end{align}
Thus, letting $\epsilon\to0$ and using Lemma \ref{largeintoftenvisited}, we get that the right-hand side of \eqref{probaxvisitsneighborhood} equals $1$. Since this holds for all $z \in (0,1)$ and $\eta>0$, we conclude that $x \in \mathcal{R}(X)$. Therefore, $Supp \ \rho^X_{\infty} \subset \mathcal{R}(X)$. 
\end{proof}

\subsection{Open set recurrence for $\Lambda$-Wright--Fischer processes}\label{applopensetreclwf}

In this Subsection we prove Theorem \ref{fullsupport}. We split the proof into two lemmas. Lemma \ref{fullsupportdiff} covers the case $\Lambda(\{0\})>0$ while  Lemma \ref{fullsupportjump} covers the case $\Lambda((0,1/2+\epsilon])>0$ for every $\epsilon>0$. 

\begin{lemme} \label{fullsupportdiff}
For any solution $X$ of \eqref{eq:SDEWFP}, if $\Lambda(\{0\})>0$ and $\mathcal{R}(X) \cap (0,1) \neq \emptyset$ then $\mathcal{R}(X)=[0,1]$, in particular $X$ is open set recurrent. 
\end{lemme}
The idea of the proof is to compose $X$ with a carefully chosen function so that the diffusion coefficient of the resulting process becomes constant (with respect to the process value). This transformation enables us to prescribe the process value within any desired interval with positive probability. A significant part of the proof is then devoted to controlling the remaining terms.

\begin{proof}[Proof of Lemma \ref{fullsupportdiff}]
We assume that $\Lambda(\{0\})>0$ and $\mathcal{R}(X) \cap (0,1) \neq \emptyset$. Let us fix $x \in \mathcal{R}(X) \cap (0,1)$ and $y \in (0,1)$. We only need to prove that $y \in \mathcal{R}(X)$. We fix $m \in (0,1)$ and $\epsilon \in (0,1/2)$ that will be determined later. We define the stopping time $T(\epsilon)$ via $$T(\epsilon)\coloneqq \inf \{ t \geq 0, \ X_t \in [\epsilon,1-\epsilon]^c \}.$$ We denote by $N_0$ and $N_1$ the restrictions of the Poisson random measure $N$ to $[0,\infty)\times(0,m]\times [0,1]$ and $[0,\infty)\times(m,1]\times [0,1]$ respectively. We note that $N_0$ and $N_1$ are independent Poisson random measures and that $N=N_0+N_1$. 

Let $w \in (\epsilon,1-\epsilon)$. We define the bijection $f:[0,1] \rightarrow [0,\pi]$ by $f(v)\coloneqq\frac{\pi}{2}+\arcsin(2v-1)$. Then, by applying It{\^o}'s formula (see for example \cite[Thm. II.5.1]{ikedawatanabe}) to a function $f_a$ that coincides with $f$ on $[a,1-a]$ (for $a \in (0,\epsilon)$ arbitrary small) and that is smooth on $\mathbb{R}$ we have $\mathbb{P}_w$-almost surely that, for all $t \in [0,T(\epsilon)]$, 
\begin{align}
f(X_t) \  &= \  f(w) + \sqrt{\Lambda(\{0\})} W_t + \mathcal{I}_{N_0}(t) + \mathcal{I}_{N_1}(t) + \mathcal{I}_{\Lambda}(t) + \mathcal{I}_c(t) + \mathcal{I}_S(t) + \mathcal{I}_M(t), \label{fxtito}
\end{align}
where 
\begin{align*}
\mathcal{I}_{N_0}(t) & \coloneqq \int_{[0,t]\times(0,m]\times [0,1]} \big ( f((1-r)X_{s-} + r\mathds{1}_{\{u\leq X_{s-}\}})-f(X_{s-}) \big)\tilde{N}(\dd s, \dd r,\dd u), \\
\mathcal{I}_{N_1}(t) & \coloneqq \int_{[0,t]\times(m,1]\times [0,1]} \big ( f((1-r)X_{s-} + r\mathds{1}_{\{u\leq X_{s-}\}})-f(X_{s-}) \big)\tilde{N}(\dd s, \dd r,\dd u), \\
\mathcal{I}_{\Lambda}(t) & \coloneqq \int_0^t \int_{(0,1]} \int_{[0,1]} \big ( f((1-r)X_{s} + r\mathds{1}_{\{u\leq X_{s}\}})-f(X_{s}) \nonumber\\
&\qquad \qquad \quad \qquad - rf'(X_s)(\mathds{1}_{\{u\leq X_{s}\}}(1-X_{s})-\mathds{1}_{\{u>X_{s}\}} X_{s}) \big) \dd u \ r^{-2} \Lambda(\dd r) \dd s, \\
\mathcal{I}_c(t) & \coloneqq \int_0^t \left ( \frac{\Lambda(\{0\})}{2} f''(X_s) X_s(1-X_s)\right.\\
&\qquad \qquad \quad \qquad\left.+ f'(X_s)(\sigma(X_s)X_s(1-X_s)+\theta_a (1-X_s)-\theta_A X_s) \right) \dd s, \\
\mathcal{I}_S(t) & \coloneqq \int_{[0,t]\times(-1,1)} \big ( f(X_{s-}+rX_{s-}(1-X_{s-}))-f(X_{s-}) \big ) S(\dd s, \dd r), \\
\mathcal{I}_M(t) & \coloneqq \int_{[0,t]\times(-1,1)} \big ( f((1-|r|)X_{s-}+\lvert r\rvert \mathds{1}_{\{r\geq 0\}})-f(X_{s-}) \big ) M(\dd s, \dd r). 
\end{align*}

\noindent Let $C_{\sigma}\coloneqq\sup_{v \in [0,1]} |\sigma(v)|$, $C_1(\epsilon)\coloneqq\sup_{v \in [\epsilon,1-\epsilon]} |f'(v)|$ and $C_2(\epsilon)\coloneqq\sup_{v \in [\epsilon,1-\epsilon]} |f''(v)|$. We have $\mathbb{P}_w$-almost surely that for all $t \in [0,T(\epsilon)]$, 
\begin{align}
|\mathcal{I}_c(t)| \leq \big (\Lambda(\{0\}) C_2(\epsilon) + C_1(\epsilon)(C_{\sigma}+\theta_a+\theta_A) \big )t, \label{controltermdrift}
\end{align}
and, using Taylor's theorem, 
\begin{align}
|\mathcal{I}_{\Lambda}(t)| & \leq C_2(\epsilon) \Lambda((0,1]) t, \label{controltermneurtalcomp} \\
|\mathcal{I}_S(t)| & \leq C_1(\epsilon) \int_{[0,t]\times(-1,1)}\!\!\!\! |r| S(\dd s, \dd r), \quad |\mathcal{I}_M(t)| \leq C_1(\epsilon) \int_{[0,t]\times(-1,1)} \!\!\!\!|r| M(\dd s, \dd r). \label{controltermenv}
\end{align}
It is easy to see that, conditionally on $W$, $S$, $M$ and $N_1$, the process $(\mathcal{I}_{N_0}(t \wedge T(\epsilon)))_{t \geq 0}$ is a martingale. By Doob's maximal inequality, the compensation formula and Taylor's theorem we have, for any $t \geq 0$, 
\begin{align}
 &\frac1{4} \mathbb{E}_w \left [ \sup_{s \in [0,t \wedge T(\epsilon)]} |\mathcal{I}_{N_0}(s)|^2 \big | W, S, M, N_1 \right ] \leq \mathbb{E}_w \left [ |\mathcal{I}_{N_0}(t \wedge T(\epsilon))|^2 \big | W, S, M, N_1 \right ] \nonumber \\
= & \mathbb{E}_w \bigg [ \int_0^t \int_{(0,m]} \int_{[0,1]} \left ( f((1-r)X_{s} + r\mathds{1}_{\{u\leq X_{s}\}})\right.\bigg.\nonumber\\&\qquad\qquad\qquad\qquad\qquad\qquad\bigg.\left. -f(X_{s}) \right)^2 \mathds{1}_{\{s\leq T(\epsilon)\}} \dd u \ r^{-2} \Lambda(\dd r) \dd s \big | W, S, M, N_1 \bigg ] \nonumber \\
\leq & C_1(\epsilon)^2 \Lambda((0,m]) t. \label{controltermneurtal}
\end{align}

We fix $\eta \in (0,\min\{x,1-x,y,1-y\})$. Let $\eta'>0$ be such that $f^{-1}((f(x)-7\eta',f(x)+7\eta'))\subset(x-\eta,x+\eta)$ and $f^{-1}((f(y)-7\eta',f(y)+7\eta'))\subset(y-\eta,y+\eta)$. Let also $\eta'' \in (0,\min\{x,1-x\})$ be such that $[x-\eta'',x+\eta''] \subset f^{-1}((f(x)-\eta',f(x)+\eta'))$. We then choose $\epsilon >0$ such that $(x-\eta \vee \eta'',x+\eta \vee \eta'') \cup (y-\eta,y+\eta) \subset (\epsilon,1-\epsilon)$. We note from the definition of $\eta''$ that 
\begin{align}
\forall v \in [x-\eta'',x+\eta''], \ |f(v)-f(x)|<\eta'. \label{localboundf}
\end{align}
From now one we assume that $w \in [x-\eta'',x+\eta'']$. We then choose 
\begin{align*}
\frac1{\delta} \coloneqq & \frac1{\eta'} \times \max \left \{ \Lambda(\{0\}) C_2(\epsilon) + C_1(\epsilon)(C_{\sigma}+\theta_a+\theta_A), \right . \\
& \qquad \qquad \left . C_2(\epsilon) \Lambda((0,1]), 2C_1(\epsilon) \mu((-1,1)), 2C_1(\epsilon) \nu((-1,1)) \right \}.  
\end{align*}
From the choice of $\delta$ and \eqref{controltermdrift}-\eqref{controltermneurtalcomp} we see that, $\mathbb{P}_w$-almost surely, $$\sup_{t \in [0,\delta \wedge T(\epsilon)]} |\mathcal{I}_c(t)| \leq \eta', \ \sup_{t \in [0,\delta \wedge T(\epsilon)]} |\mathcal{I}_{\Lambda}(t)| \leq \eta'.$$
Moreover, we have by the compensation formula that 
\begin{align}
\mathbb{E} \left [ \int_{[0,\delta]\times(-1,1)} |r| S(\dd s, \dd r) \right ] & = \delta \mu((-1,1))<\frac{\eta'}{C_1(\epsilon)}, \label{espenv1} \\ \mathbb{E} \left [ \int_{[0,\delta]\times(-1,1)} |r| M(\dd s, \dd r) \right ] & = \delta \nu((-1,1))<\frac{\eta'}{C_1(\epsilon)}. \label{espenv2}
\end{align}
We consider the event 
\begin{align*}
\mathcal{E} \coloneqq & \left \{ \sup_{t \in [0,\delta]} \left |\sqrt{\Lambda(\{0\})} W_t -(f(y)-f(x))\frac{t}{\delta} \right |<\eta' \right \} \cap \{ N_1([0,\delta]\times(m,1]\times [0,1])=0 \} \\
& \qquad \cap \left \{ \int_{[0,\delta]\times(-1,1)} |r| S(\dd s, \dd r)<\frac{\eta'}{C_1(\epsilon)} \right \} \cap \left \{ \int_{[0,\delta]\times(-1,1)} |r| M(\dd s, \dd r)<\frac{\eta'}{C_1(\epsilon)} \right \}. 
\end{align*}
The first event has positive probability since every open set, in the space of continuous functions on $[0,\delta]$, has a positive weight for the Wiener measure. The second event has probability $e^{-\delta \int_{(m,1]} r^{-2} \Lambda(\dd r)}$. We see from \eqref{espenv1}-\eqref{espenv2} that the third and fourth events have a positive probability. Since $W$, $S$, $M$ and $N_1$ are independent, the four events in the intersection defining $\mathcal{E}$ are independent. We deduce that $\mathcal{E}$ has positive probability that we denote $p(\delta,m,\eta',x,y)$ and that does not depend on the starting point $w$ of the diffusion. We see from the above and \eqref{controltermenv} that, on this event, we have $\mathbb{P}_w$-almost surely 
\begin{align}
\forall i \in \{ \Lambda, c, S, M \}, \ \sup_{t \in [0,\delta \wedge T(\epsilon)]} |\mathcal{I}_i(t)| \leq \eta' \qquad \text{and} \qquad \sup_{t \in [0,\delta \wedge T(\epsilon)]} |\mathcal{I}_{N_1}(t)|=0. \label{controltermsend}
\end{align}
Then we choose $m\in(0,1)$ such that $4\eta'^{-2} C_1(\epsilon)^2 \Lambda((0,m]) \delta<1/2$. Let us set 
\begin{align*}
\mathcal{E}'\coloneqq \left \{\sup_{t \in [0,\delta \wedge T(\epsilon)]} |\mathcal{I}_{N_0}(t)|\leq\eta' \right \}. 
\end{align*}
By Markov inequality, \eqref{controltermneurtal} and the choice of $m$ we get 
\begin{align}
\mathbb{P}_w ( \mathcal{E}'^{c} \big | W, S, M, N_1 ) & \leq \eta'^{-2} \mathbb{E}_w \left [ \sup_{s \in [0,\delta \wedge T(\epsilon)]} |\mathcal{I}_{N_0}(t)|^2 \big | W, S, M, N_1 \right ] \nonumber \\
& \leq 4\eta'^{-2} C_1(\epsilon)^2 \Lambda((0,m]) \delta < 1/2. \label{estimprobacondsmallneutral0}
\end{align}
We set $\mathcal{E}'' \coloneqq \mathcal{E} \cap \mathcal{E}'$. Recalling that $\mathbb{P}_w(\mathcal{E})=p(\delta,m,\eta',x,y)$ and using \eqref{estimprobacondsmallneutral0} together with the fact that $\mathcal{E}$ is measurable with respect to the sigma-field generated by $W,S,M,N_1$, we get  that, for any $w \in [x-\eta'',x+\eta'']$, 
\begin{align}
\mathbb{P}_w (\mathcal{E}'') = \mathbb{E}_w \left [ \mathds{1}_{\mathcal{E}} \mathbb{E}_w \left [ \mathds{1}_{\mathcal{E}'} \big | W, S, M, N_1 \right ] \right ] \geq \mathbb{P}_w(\mathcal{E})/2 = p(\delta,m,\eta',x,y)/2 >0. \label{minopeprimeprime0}
\end{align}
Combining the definitions of $\mathcal{E}'$ and $\mathcal{E}''$ with \eqref{controltermsend} we get that, on $\mathcal{E}''$, we have 
\begin{align}
\forall i \in \{ N_0, \Lambda, c, S, M \}, \ \sup_{t \in [0,\delta \wedge T(\epsilon)]} |\mathcal{I}_i(t)| \leq \eta' \qquad \text{and} \qquad \sup_{t \in [0,\delta \wedge T(\epsilon)]} |\mathcal{I}_{N_1}(t)|=0. \label{controlallterms}
\end{align}

Let us assume that $y>x$. Since $\mathcal{E}'' \subset \mathcal{E}$ we see that, on $\mathcal{E}''$, for any $t \in [0,\delta \wedge T(\epsilon)]$, 
\begin{align*}
-\eta' < \sqrt{\Lambda(\{0\})} W_t < f(y)-f(x) + \eta'. 
\end{align*}
Combining this with \eqref{localboundf} and \eqref{controlallterms} and plugging into \eqref{fxtito} we get that, for any $w \in [x-\eta'',x+\eta'']$, on $\mathcal{E}''$, we have $\mathbb{P}_w$-almost surely that, for any $t \in [0,\delta \wedge T(\epsilon)]$, $f(x) -7\eta' < f(X^1_t) < f(y) +7\eta'$. By the choices of $\eta'$ and $\epsilon$ this implies that $X_{\delta \wedge T(\epsilon)} \in (\epsilon,1-\epsilon)$ so, by almost sure right continuity, $\delta \wedge T(\epsilon) \neq T(\epsilon)$ and we get that $\delta<T(\epsilon)$. In conclusion we get that, for any $w \in [x-\eta'',x+\eta'']$, on $\mathcal{E}''$, we have $\mathbb{P}_w$-almost surely $\delta<T(\epsilon)$. This yields in particular that \eqref{fxtito} and \eqref{controlallterms} can be applied at $t=\delta$. In the case where $y\leq x$, this is proved similarly. 

Recall that we have $|\sqrt{\Lambda(\{0\})} W_\delta -(f(y)-f(x))|<\eta'$ on $\mathcal{E}''$. Plugging this, \eqref{localboundf} and \eqref{controlallterms} into \eqref{fxtito} applied at $t=\delta$ we conclude that, for any $w \in [x-\eta'',x+\eta'']$, we have $\mathbb{P}_w$-almost surely on $\mathcal{E}''$ that $|f(X_{\delta})-f(y)|<7\eta'$ and, by the choice of $\eta'$, the latter implies that $X_{\delta} \in (y-\eta,y+\eta)$. In conclusion we obtain 
\begin{align}
H \coloneqq & \inf_{w \in [x-\eta'',x+\eta'']} \mathbb{P}_w \left ( X_{\delta} \in (y-\eta,y+\eta) \right ) \geq \inf_{w \in [x-\eta'',x+\eta'']} \mathbb{P}_w(\mathcal{E}'') >0, \label{minounifproba}
\end{align}
where the last bound come from \eqref{minopeprimeprime0}. 
Now let us fix $z \in (0,1)$ and set $T_1\coloneqq \inf \{ t \geq 0, \ X_t \in (x-\eta'',x+\eta'') \}$ and for any $n \geq 2$, $T_n\coloneqq \inf \{ t \geq T_{n-1}+1, \ X_t \in (x-\eta'',x+\eta'') \}$. Since $x \in \mathcal{R}(X)$, we have $\mathbb{P}_z$-almost surely that $T_n<\infty$ for all $n \geq 1$. By the strong Markov property of $X$ at $T_n$, 
\begin{align}
\mathbb{P}_z \left ( X_{T_n+\delta} \in (y-\eta,y+\eta) \right ) \geq H. \label{probareachtransfo}
\end{align}
We get by the Markov property that, $\mathbb{P}_z$-almost surely, the $\eta$-neighbourhood of $y$ is visited by $X$ at arbitrary large times. Since this holds for all $\eta>0$ small enough and $z \in (0,1)$ we obtain $y \in \mathcal{R}(X)$. This concludes the proof. 
\end{proof}

\begin{lemme} \label{fullsupportjump}
For any solution $X$ of \eqref{eq:SDEWFP}, if $\Lambda((0,1/2+\epsilon])>0$ for every $\epsilon>0$ and $\mathcal{R}(X) \cap (0,1) \neq \emptyset$ then $\mathcal{R}(X)=[0,1]$, in particular $X$ is open set recurrent. 
\end{lemme}
The first part of the proof deals with neutral events of size approximately $r$ (for $r \in (0,1/2] \cap Supp \ \Lambda$) and shows that, if $x \in \mathcal{R}(X) \cap (0,1)$, then $x+r(1-x) \in \mathcal{R}(X)$ and $(1-r)x \in \mathcal{R}(X)$. The second part of the proof shows that this stability property of $\mathcal{R}(X)$, together with its closedness, suffices to ensure that $\mathcal{R}(X)=[0,1]$. 
\begin{proof}[Proof of Lemma \ref{fullsupportjump}]
We assume that $\Lambda((0,1/2+\epsilon])>0$ for every $\epsilon>0$ and $\mathcal{R}(X) \cap (0,1) \neq \emptyset$. If $\Lambda(\{0\})>0$ then the result follows from Lemma \ref{fullsupportdiff}. We can thus assume $\Lambda(\{0\})=0$. Since $\Lambda((0,1/2+\epsilon])>0$ for every $\epsilon>0$, $(0,1/2+\epsilon]$ contains a point of $Supp \ \Lambda$ for every $\epsilon>0$ (see section \ref{notations}). Since $Supp \ \Lambda$ is a closed set (see section \ref{notations}) we deduce that $(0,1/2]$ contains a point of $Supp \ \Lambda$. We thus fix $r \in (0,1/2] \cap Supp \ \Lambda$. Let us choose $x \in \mathcal{R}(X) \cap (0,1)$, $\eta \in (0,\min\{x/5,(1-x)/5,r\})$. Let $m \in (0,r-\eta)$ that will be determined later. Again, we denote by $N_0$ and $N_1$ the restrictions of the Poisson random measure $N$ to $[0,\infty)\times(0,m]\times [0,1]$ and $[0,\infty)\times(m,1]\times [0,1]$ respectively. We note that $N_0$ and $N_1$ are independent Poisson random measures and that $N=N_0+N_1$. Since $\Lambda(\{0\})=0$, we have by \eqref{eq:SDEWFP} that $\mathbb{P}_w$-almost surely, for all $t \geq 0$, 
\begin{align}
X_t \  &= \  w + \mathcal{J}_{N_0}(t) + \mathcal{J}_{N_1}(t) + \mathcal{J}_c(t) + \mathcal{J}_S(t) + \mathcal{J}_M(t), \label{xtermssde}
\end{align}
where 
\begin{align*}
%\mathcal{J}_W(t) & \coloneqq \int_0^t \sqrt{\Lambda(\{0\})X_s(1-X_s)} \dd W_s, \\
\mathcal{J}_{N_0}(t) & \coloneqq \int_{[0,t]\times(0,m]\times [0,1]} r\big(\mathds{1}_{\{u\leq X_{s-}\}}(1-X_{s-})-\mathds{1}_{\{u>X_{s-}\}} X_{s-} \big)\tilde{N}(\dd s, \dd r,\dd u), \\
\mathcal{J}_{N_1}(t) & \coloneqq \int_{[0,t]\times(m,1]\times [0,1]} r\big(\mathds{1}_{\{u\leq X_{s-}\}}(1-X_{s-})-\mathds{1}_{\{u>X_{s-}\}} X_{s-} \big)\tilde{N}(\dd s, \dd r,\dd u), \\
\mathcal{J}_c(t) & \coloneqq \int_0^t \big ( \sigma(X_s)X_s(1-X_s)+\theta_a (1-X_s)-\theta_A X_s \big ) \dd s, \\
\mathcal{J}_S(t) & \coloneqq \int_{[0,t]\times(-1,1)} rX_{s-}(1-X_{s-}) S(\dd s, \dd r), \\
\mathcal{J}_M(t) & \coloneqq \int_{[0,t]\times(-1,1)} \lvert r\rvert (\mathds{1}_{\{r\geq 0\}}(1-X_{t-})-\mathds{1}_{\{r< 0\}}X_{t-}) M(\dd s, \dd r). 
\end{align*}

Let $C_{\sigma}\coloneqq\sup_{v \in [0,1]} |\sigma(v)|$. We have $\mathbb{P}_w$-almost surely that for all $t \geq 0$, 
\begin{align}
|\mathcal{J}_c(t)| & \leq (C_{\sigma}+\theta_a+\theta_A) t, \qquad |\mathcal{J}_S(t)| \leq \int_{[0,t]\times(-1,1)} |r| S(\dd s, \dd r), \nonumber\\
|\mathcal{J}_M(t)| & \leq \int_{[0,t]\times(-1,1)} |r| M(\dd s, \dd r). \label{controltermenvsde}
\end{align}
It is easy to see that, conditionally on $S$, $M$ and $N_1$, the process $(\mathcal{J}_{N_0}(t))_{t \geq 0}$ is a martingale. By Doob's maximal inequality and the compensation formula we have, for any $t \geq 0$, 
\begin{align}
& \frac1{4} \mathbb{E}_w \left [ \sup_{s \in [0,t]} |\mathcal{J}_{N_0}(s)|^2 \big | S, M, N_1 \right ] \leq \mathbb{E}_w \left [ |\mathcal{J}_{N_0}(t)|^2 \big | S, M, N_1 \right ] \nonumber \\
= & \mathbb{E}_w \left [ \int\limits_0^t \int\limits_{(0,m]} \int\limits_{[0,1]} r^2 \big(\mathds{1}_{\{u\leq X_{s-}\}}(1-X_{s-})-\mathds{1}_{\{u>X_{s-}\}} X_{s-} \big)^2 \dd u \ r^{-2} \Lambda(\dd r) \dd s \mid  S, M, N_1 \right ] \nonumber \\
\leq & \Lambda((0,m]) t. \label{controltermneurtalsde}
\end{align}
%Then, using Doob's maximal inequality and It{\^o}'s we also have, for any $t \geq 0$, 
%\begin{align}
%& \frac1{4} \mathbb{E}_w \left [ \sup_{s \in [0,t]} |\mathcal{J}_W(s)|^2 \right ] \leq \mathbb{E}_w \left [ |\mathcal{J}_W(t)|^2 \right ] \leq \mathbb{E}_w \left [ \int_0^t \Lambda(\{0\})X_s(1-X_s) \dd s \right ] \leq \Lambda(\{0\}) t. \label{controltermneurtal}
%\end{align}
From now one we assume that $w \in [x-\eta,x+\eta]$. We then choose 
\begin{align*}
\frac1{\delta} \coloneqq & \frac1{\eta} \times \max \left \{ C_{\sigma}+\theta_a+\theta_A, 2\mu((-1,1)), 2\nu((-1,1)) \right \}.  
\end{align*}
Then we see from the choice of $\delta$ and \eqref{controltermenvsde} that, $\mathbb{P}_w$-almost surely, $\sup_{t \in [0,\delta]} |\mathcal{J}_c(t)| \leq \eta$. Moreover, by the compensation formula that, we have 
\begin{align}
\mathbb{E} \left [ \int_{[0,\delta]\times(-1,1)} |r| S(\dd s, \dd r) \right ] &= \delta \mu((-1,1)) <\eta,\nonumber \\
\ \mathbb{E} \left [ \int_{[0,\delta]\times(-1,1)} |r| M(\dd s, \dd r) \right ] &= \delta \nu((-1,1)) <\eta. \label{espsmall}
\end{align}
We consider the event 
\begin{align*}
\mathcal{E} \coloneqq & \left \{ \int_{[0,\delta]\times(-1,1)} |r| S(\dd s, \dd r)<\eta \right \} \cap \left \{ \int_{[0,\delta]\times(-1,1)} |r| M(\dd s, \dd r)<\eta \right \} \\
& \cap \{ N_1([0,\delta]\times(m,1]\times [0,1])=N_1([0,\delta]\times(r-\eta,r+\eta) \times (0,x-5\eta))=1 \}. 
\end{align*}
We see from \eqref{espsmall} that the first two events have a positive probability. The third event says that the Poisson random measure $N$ has exactly one atom in $[0,\delta]\times(m,1]\times [0,1]$ and that this atom lies into $[0,\delta]\times(r-\eta,r+\eta) \times (0,x-5\eta)$. On $\mathcal{E}$ we denote this atom by $(S,R,U)$. Since the intensity measure $\dd t\times r^{-2}\Lambda(\dd r)\times \dd u$ give a finite weight to $[0,\delta]\times(m,1]\times [0,1]$ and a positive weight to $[0,\delta]\times(r-\eta,r+\eta) \times (0,x-5\eta)$ (because $r \in Supp \ \Lambda$), we see that the third event also has positive probability. Since the Poisson random measures $S$, $M$ and $N_1$ are independent, the three events in the intersection defining $\mathcal{E}$ are independent. We deduce that $\mathcal{E}$ has positive probability that we denote $p(\delta,m,r,\eta,x)$ and that does not depend on the starting point $w$ of the diffusion. We see from the above and \eqref{controltermenvsde} that, on this event, we have $\mathbb{P}_w$-almost surely 
\begin{align}
\forall i \in \{ c, S, M \}, \ \sup_{t \in [0,\delta]} |\mathcal{J}_i(t)| \leq \eta. \label{controltermsendsde}
\end{align}
Choose $m\in(0,1)$ such that $4\eta'^{-2} \Lambda((0,m]) \delta<1/2$ and set $\mathcal{E}'\coloneqq \{\sup_{t \in [0,\delta]} |\mathcal{J}_{N_0}(t)|\leq\eta \}$. By Markov inequality, \eqref{controltermneurtalsde} and the choice of $m$, we get 
\begin{align}
\mathbb{P}_w ( \mathcal{E}'^{c} \big | S, M, N_1 ) \leq \eta'^{-2} \mathbb{E}_w \left [ \sup_{s \in [0,\delta]} |\mathcal{J}_{N_0}(t)|^2 \big | S, M, N_1 \right ] \leq 4\eta'^{-2} \Lambda((0,m]) \delta < 1/2. \label{estimprobacondsmallneutral}
\end{align}
We set $\mathcal{E}'' \coloneqq \mathcal{E} \cap \mathcal{E}'$. Recalling that $\mathbb{P}_w(\mathcal{E})=p(\delta,m,r,\eta,x)$ and using \eqref{estimprobacondsmallneutral} together with the fact that $\mathcal{E}$ is measurable with respect to the sigma-field generated by $S,M,N_1$, we get  that, for any $w \in [x-\eta,x+\eta]$, 
\begin{align}
\mathbb{P}_w (\mathcal{E}'') = \mathbb{E}_w \left [ \mathds{1}_{\mathcal{E}} \mathbb{E}_w \left [ \mathds{1}_{\mathcal{E}'} \big | S, M, N_1 \right ] \right ] \geq \mathbb{P}_w(\mathcal{E})/2 = p(\delta,m,r,\eta,x)/2 >0. \label{minopeprimeprime}
\end{align}
Combining the definitions of $\mathcal{E}'$ and $\mathcal{E}''$ with \eqref{controltermsendsde} we get that, on $\mathcal{E}''$, we have 
\begin{align}
\forall i \in \{ c, N_0, S, M \}, \ \sup_{t \in [0,\delta]} |\mathcal{J}_i(t)| \leq \eta. \label{controlalltermssde}
\end{align}
Recall that, on the event $\mathcal{E}$ (that contains $\mathcal{E}''$), $(S,R,U)$ denotes the unique atom of $N_1$ and that $(S,R,U) \in [0,\delta]\times(r-\eta,r+\eta) \times (0,x-5\eta)$. By the definition of $(S,R,U)$ we see that, on $\mathcal{E}''$, we have $\mathcal{J}_{N_1}(S-)=0$. Combining with \eqref{xtermssde} and \eqref{controlalltermssde} we get that for any $w \in [x-\eta,x+\eta]$ we have $\mathbb{P}_w$-almost surely on $\mathcal{E}''$ that $x+5\eta \geq X_{S-} \geq x-5\eta > U$. Combining this with the fact that $R \in (r-\eta,r+\eta)$ we get that $\mathcal{J}_{N_1}(\delta)=R(1-X_{S-}) \in (r(1-x)-6\eta, r(1-x)+6\eta)$. Then using \eqref{xtermssde} and \eqref{controlalltermssde} again we get that, for any $w \in [x-\eta,x+\eta]$, we have $\mathbb{P}_w$-almost surely on $\mathcal{E}''$ that $X_{\delta} \in (x+r(1-x)-11\eta,x+r(1-x)+11\eta)$. Summarizing, we obtain 
\begin{align}
\inf_{w \in [x-\eta,x+\eta]} \mathbb{P}_w \left ( X_{\delta} \in (x+r(1-x)-11\eta,x+r(1-x)+11\eta) \right ) & \geq \inf_{w \in [x-\eta,x+\eta]} \mathbb{P}_w(\mathcal{E}'') >0, \label{minounifprobasde}
\end{align}
where the last bound come from \eqref{minopeprimeprime}. Since $x \in \mathcal{R}(X)$, we can choose $\eta$ arbitrary small and proceed as in the proof of Lemma \ref{fullsupportdiff} to obtain $x+r(1-x) \in \mathcal{R}(X)$. 

Now, set  $f^r_0(x)\coloneqq x+r(1-x)$. We have shown that, if $x \in \mathcal{R}(X)$, then $f^r_0(x) \in \mathcal{R}(X)$. Replacing $(0,x-5\eta)$ by $(x+5\eta,1)$ in the definition of $\mathcal{E}$, we obtain similarly that, if $x \in \mathcal{R}(X)$, then $f^r_1(x) \in \mathcal{R}(X)$, where $f^r_1(x)\coloneqq(1-r)x$. In other words, $\mathcal{R}(X)$ enjoys the following stability property: for any $x \in (0,1)$, 
\begin{align} \label{stabilityprop}
x \in \mathcal{R}(X) \Rightarrow \left\{
\begin{aligned}
& f^r_0(x) \in \mathcal{R}(X), \\
& f^r_1(x) \in \mathcal{R}(X). \end{aligned} \right. 
\end{align}
Let us denote by $\psi : [0,1] \rightarrow [-1/2,1,2]$ the bijection given by $\psi(x)\coloneqq x-1/2$. We define the transformations $F^r_0, F^r_1 : [-1/2,1,2] \rightarrow [-1/2,1,2]$ by $F^r_i(y) \coloneqq \psi \circ f^r_i \circ \psi^{-1}$ for $i \in \{0,1\}$ and we note that, for any $y \in [-1/2,1,2]$, $F_i(y)=y(1-r)+(-1)^ir/2$. Therefore, for any $n \geq 1$, $(i_k)_{0 \leq k \leq n-1} \in \{0,1\}^n$ and $x \in (0,1)$ we have 
\begin{align*}
\psi \circ f^r_{i_1} \circ \ldots \circ f^r_{i_n}(x) = F^r_{i_0} \circ \ldots \circ F^r_{i_{n-1}} \circ \psi(x) = \frac{r}{2} \sum_{k=0}^{n-1}(-1)^{i_k}(1-r)^{k}+\psi(x)(1-r)^n. 
\end{align*}
Using the stability property \eqref{stabilityprop}, the closedness of $\mathcal{R}(X)$, and that $\psi$ is an homeomorphism, we deduce that, if $\mathcal{R}(X) \cap (0,1) \neq \emptyset$ then for any $(i_k)_{k \geq 0} \in \{0,1\}^{\mathbb{N}}$ we have 
\begin{align}
\frac{r}{2} \sum_{k=0}^{\infty}(-1)^{i_k}(1-r)^{k} \in \psi(\mathcal{R}(X)). \label{sumofseriesonpsirzx}
\end{align}
To reach a target value $u \in [-1/2,1,2]$ as the sum of a series as in \eqref{sumofseriesonpsirzx}, we apply an upward and downward procedure that consists in adding positive terms to the series while we are below the target value, then adding negative terms while we are above the target value and so one. This procedure succeeds to converge to the target value if the absolute value of each term is smaller than the sum of the absolute values of the remaining terms. Since $r \in (0,1/2]$ (this is where this fact is crucially used) we have for any $j \geq 0$ that $(1-r)^j \leq \frac{1-r}{r}(1-r)^j = \sum_{k=j+1}^{\infty}(1-r)^{k}$. By the previous discussion, this allows, for any $u \in [-1/2,1,2]$, to choose $(i_k)_{k \geq 0} \in \{0,1\}^{\mathbb{N}}$ so that the left-hand side of \eqref{sumofseriesonpsirzx} equals $u$. We thus get $\psi(\mathcal{R}(X))=[-1/2,1,2]$ and therefore $\mathcal{R}(X)=[0,1]$, which concludes the proof. 
\end{proof}

\subsection{Optimality of the condition from Theorem \ref{fullsupport}} \label{optimalityofcondition}

In this subsection we consider measures $\Lambda\in\Ms_f([0,1])$ that do not satisfy the condition from Theorem \ref{fullsupport} and we build examples of $X$ solving \eqref{eq:SDEWFP} (with these $\Lambda$) with $\mathcal{R}(X) \cap (0,1) \neq \emptyset$ and $1/2 \notin \mathcal{R}(X)$. This will prove the optimality of the condition, as stated in Remark \ref{remarkoptimality}. 

Let $\Lambda\in\Ms_f([0,1])$ be a non-zero measure that does not satisfy the condition in Theorem \ref{fullsupport}. This implies the existence of $m\in(1/2,1)$ such that $\Lambda([0,m])=0$. We fix such an $m$ and set the other parameters in \eqref{eq:SDEWFP} as follows: $\mu \equiv 0$, $\theta=(0,0)$, $\sigma \equiv 0$, and we choose $\nu\in\Ms_f([-1,1])$ such that $\nu([-m,m])=0$, $\nu((m,1))>0$ and $\nu((-1,-m))>0$. Then \eqref{eq:SDEWFP} becomes 
\begin{align}
\dd X_t \  & = \int_{(m,1]\times [0,1]} \! \! \! r\big(\mathds{1}_{\{u\leq X_{t-}\}}(1-X_{t-})-\mathds{1}_{\{u>X_{t-}\}} X_{t-} \big)\tilde{N}(\dd t, \dd r,\dd u)\nonumber\\
& + \int_{(-1,-m)\cup(m,1)} \lvert r\rvert (\mathds{1}_{\{r\geq 0\}}(1-X_{t-})-\mathds{1}_{\{r< 0\}}X_{t-}) M(\dd t, \dd r). 
\label{eq:SDEcontreex}
\end{align}
With this parameter choice, the coefficient $C$ from \eqref{exporate} equals $\nu((m,1])+\nu([-1,-m))$ and is thus positive. We infer from Theorem \ref{thbounddiststa} that $X$ is $W_1$-uniformly ergodic on $[0,1]$ and, by Proposition \ref{suppofmuinftyisrecunif}, we deduce that $\mathcal{R}(X) \neq \emptyset$.

If $\mathcal{R}(X) \cap (0,1)=\emptyset$, then $\mathcal{R}(X) \subset \{0,1\}$. In this case, if $0 \in \mathcal{R}(X)$, then we choose $r \in (m,1) \cap Supp \ \nu$ (this is possible since $\nu((m,1))>0$; see section \ref{notations}) and, proceeding as in the proof of Lemma \ref{fullsupportjump}, one can show that $r=0+r(1-0) \in \mathcal{R}(X)$. If $1 \in \mathcal{R}(X)$, then we choose $r \in (-1,-m) \cap Supp \ \nu$ and get that $1+r=1-|r|\times 1 \in \mathcal{R}(X)$. In both cases, we get a contradiction with $\mathcal{R}(X) \subset \{0,1\}$. Therefore, $\mathcal{R}(X) \cap (0,1) \neq \emptyset$. 

Furthermore, from \eqref{eq:SDEcontreex}, we see that the transitions of the process $X$ occur at finite rate, and each transition is either of the form $X_t = X_{t-} + h(1-X_{t-})=h+X_{t-}(1-h)$ for some $h \in (m,1)$, or of the form $X_t = X_{t-} - hX_{t-}=(1-h)X_{t-}$ for some $h \in (m,1)$. Right after the first type of transition, we have almost surely $X_t>m$, and right after the second type of transition, we have almost surely $X_t<1-m$. In particular, if $T$ denotes the first transition time we have almost surely $X_t \in [0,1-m] \cup [m,1]$ for all $t \geq T$. In particular, $X$ never visits $(1-m,m)$ after time $T$. We conclude that $1/2 \notin \mathcal{R}(X)$ as desired. 

\begin{appendix}
\section{Construction of the Bernstein dual of a $\Lambda$-Wright--Fisher process} \label{constructiondual}

As announced in Section~\ref{sec:model}, we now define the Bernstein dual \( V \coloneqq (V_t)_{t \geq 0} \) of a solution $X$ to the SDE~\eqref{eq:SDEWFP}, where $\sigma = \sigma_s$ is given as in~\eqref{defselectionterm}. For our purposes, we present a construction that differs slightly - though is mathematically equivalent - from the one in~\cite{cordhumvech2024}. This alternative construction is chosen because it facilitates the derivation of a monotone coupling for the Bernstein dual (see Lemma~\ref{monotonecouplingASG} below), which is instrumental in establishing the open set recurrence property of $X$ in Theorem~\ref{thbounddiststa}.

Let $N_0$, $S_0$ and $M_0$ be Poisson random measures on, respectively, $(0,\infty) \times (0,1) \times (\{0,1\}^{\mathbb{N}})$, $(0,\infty) \times (-1,1) \times (\{0,1\}^{\mathbb{N}})$ and $(0,\infty) \times (-1,1) \times (\{0,1\}^{\mathbb{N}})$, with intensity measures respectively $\mathcal{M}_N$, $\mathcal{M}_S$ and $\mathcal{M}_M$, respectively, given by

\begin{align*}
\mathcal{M}_N(\dd s, \dd r, \dd z)&\coloneqq\dd s (\berno{r}^{\times \mathbb{N}}(\dd z)) r^{-2} \Lambda(\dd r), \\
\mathcal{M}_S(\dd s, \dd r, \dd z)&\coloneqq\dd s (\berno{r}^{\times \mathbb{N}}(\dd z)) |r|^{-1}\mu(\dd r), \\
\mathcal{M}_M(\dd s, \dd r, \dd z)&\coloneqq\dd s (\berno{r}^{\times \mathbb{N}}(\dd z)) |r|^{-1}\nu(\dd r), 
\end{align*}
where we write $\berno{r}$ for the distribution of a Bernoulli random variable with parameter $r$. We assume that $N_0$, $S_0$ and $M_0$ are mutually independent. If there are currently $n$ lines in the system, the following transitions occur.

\smallskip

\noindent (a) \emph{Coalescences (neutral events).} At any jump $(s,r,(z_i)_{i\geq 1}) \in N_0$ satisfying that $\sharp \{ i \in [n], \ z_i=1 \} \geq 2$, each line with an index $i$ such that $z_i=1$ is erased, except the one with smallest index (the \textit{surviving line}).\\
(b) \emph{Selective branching.} Every line independently splits into $\ell$ at rate~$\beta_\ell$, that is, there are $\ell-1$ additional lines.\\
(c) \emph{Coordinated branching.} At any jump $(s,r,(z_i)_{i\geq 1}) \in S_0$ such that $\sharp \{ i \in [n], \ z_i=1 \} \geq 1$, each line with an index $i$ such that $z_i=1$ splits into two lines, one continuing and one incoming. \\
(d) \emph{Single mutation.} For $c\in\{a,A\}$, every line, independently of each other, is subject to a type-$c$ mutation at rate $\theta_c$. The line ends in a circle; the circle is white for $c=a$ and black for $c=A$. \\
(e) \emph{Coordinated mutation.} At any jump $(s,r,(z_i)_{i\geq 1}) \in M_0$ such that $\sharp \{ i \in [n], \ z_i=1 \} \geq 1$, each line with an index $i$ such that $z_i=1$ ends in a circle; the circle is white if $r>0$ and black if $r<0$.\\

The branching-coalescing structure arising under the above dynamics will be referred to as the \emph{Ancestral selection graph} (ASG). It can be shown that the transitions of the ASG occur at finite rate, and that the corresponding transition rates coincide with those described in \cite[Sec. 3.2.1]{cordhumvech2024}. We denote by $L_r$ the number of lines in the ASG at time $r$. The process $(L_r)_{r \geq 0}$ is a Markov process with transition rates identical to those specified in \cite[Sec. 3.2.1]{cordhumvech2024}. 

Let $T>0$. For a given realization of the ASG on $[0,T]$, the individuals at time~$r=0$ from which lines originate, are referred to as \emph{sinks}, while the end lines at time~$r=T$ are called \emph{sources}. We assign types (either~$a$ or~$A$) to the sources of the ASG and propagate these types along the graph towards the sinks, following the \emph{propagation rules}. The first rule states that an individual may change its type only at a transition time of the ASG. The remaining rules are as follows.
\smallskip

\noindent (a) At a coalescent event, the types of the involved lines just before the event is the type of the surviving line just after the event.\\
(b) In an selective branching, if exactly $i$ of the $\ell$ lines resulting from the branching are of type $a$, the type of the splitting line just before the branching is $a$ with probability $p_i^{(\ell)}$ and~$A$ with probability $1-p_i^{(\ell)}$.\\
(c) In a coordinated branching associated to a jump $(s,r,(z_i)_{i\geq 1}) \in S_0$ with $r>0$ (resp. $r<0$), the type of each splitting line just before the branching is the type of the associated incoming line if the latter has type $a$ (resp. $A$) and, otherwise, it is the type of the continuing line. \\ 
(d)$\&$(e) The type of an individual after a white (resp. black) circle is $a$ (resp. $A$). 

\smallskip

Now, let $P_T(x)$ be the probability, conditionally on a given realization of the ASG, that all the sinks in the ASG are of type $a$ if each source in the ASG is of type $a$ or $A$ with probability $x$ or $1-x$, respectively. 
The latter is referred to in \cite{Cordero2022,cordhumvech2024} as the \emph{Ancestral Selection Polynomial} and it can be expressed as
$$P_T(x)=\sum_{i=0}^{L_T} V_T(i)\,\binom{L_T}{i} x^i(1-x)^{L_T-i}=H(x,V_T),\qquad x\in[0,1],$$
where $V_T(i)$ then represents the conditional probability given the realization of the ASG, that all the sinks in the ASG are of type $a$, if $i$ sources are of type $a$, and $L_T-i$ sources are of type $A$. The coefficient process $(V_r)_{r\geq 0}$ has a transition at each transition time of the ASG. We refer to \cite[Section 2.5]{Cordero2022} for the effect of (a) coalescences and (b) selective branchings, and to \cite[Sec. 3.2.2]{cordhumvech2024} for the effect of (c) coordinated selection and (d)\&(e) mutations. 

\smallskip
The above construction of the ASG enjoys the following restriction property. Let $n,m$ with $m \geq n \geq 1$. In the ASG starting with $m$ lines, colour in grey the last $m-n$ lines from time $r=0$, and all lines that arise from branchings of grey lines. If two grey (resp. non-grey) lines coalesce, we colour the resulting line to be grey (resp. non-grey). If a grey line coalesces with a non-grey line, we set the resulting line to be non-grey. Then the system of non-grey lines is a realization of the ASG starting with $n$ lines. This procedure yields a coupling of two realizations of the ASG starting with respectively $m$ and $n$ lines. Moreover, we can see from the propagation rules that, for the realization of the ASG starting with $n$ lines that is contained in the ASG starting with $m$ lines, the types are not influenced by the grey lines. In particular, conditionally on a realization of the above coupling, if each source is of type $a$ or $A$ with probability $x$ or $1-x$, the probability that all the sinks in the ASG starting with $m$ lines are of type $a$ is smaller than the probability that all the sinks in the ASG starting with $n$ lines are of type $a$. Consequently, we obtain the following the following result. 
\begin{lemme}[Monotone coupling of the Bernstein dual] \label{monotonecouplingASG}
Let $e_k \in \R^{k+1}$ be given by $e_k(i)\coloneqq \mathds{1}_{\{i=k\}}$. For any $n,m$ with $m \geq n \geq 1$, there is a process $(V^m_r,V^n_r)_{r\geq 0}$, where $(V^m_r)_{r\geq 0}$ and $(V^n_r)_{r\geq 0}$ are copies of the process $(V_r)_{r\geq 0}$ defined above, starting from $e_m$ and $e_n$, respectively, and we have almost surely $H(x,V^m_r) \leq H(x,V^n_r)$ for all $x\in[0,1]$ and $r\geq 0$. 
\end{lemme}
Let $\tau \coloneqq \inf \{ s \geq 0, L_s=0 \}$. If $(L_t)_{t\geq 0}$ is absorbed at $0$ in finite time almost surely then, for any $v \in \R^\infty$, $\P_{v}$-almost surely $\tau<\infty$ and, by Condition \ref{struct}, there is a random variable $U_\infty$ such that, for $t \geq \tau$, we have $\dim(V_t)=1$ and $V_t$ is the vector whose single entry is $U_\infty$. In particular, for any $x \in [0,1]$, $H(x,V_t)=U_\infty$ for all $t \geq \tau$. 
\begin{cor} \label{m2smallerthanm1}
If $(L_t)_{t\geq 0}$ is absorbed at $0$ in finite time almost surely, then $\E_{e_2}[ U_\infty]<\E_{e_1}[ U_\infty]$. 
\end{cor}
\begin{proof}
 Let us fix $x\in [0,1]$ and consider the coupling $(V^2_r,V^1_r)_{r\geq 0}$ from Lemma \ref{monotonecouplingASG}. This coupling is obtained by a realization of the ASG that contains two lines at time $r=0$, one grey and one non-grey. We see that there is a random couple $(U^2_\infty,U^1_\infty)$ such that, almost surely, we have $H(x,V^2_r)=U^2_\infty$ and $H(x,V^1_r)=U^1_\infty$ for all $r$ sufficiently large. By Lemma \ref{monotonecouplingASG} we have almost surely $U^2_\infty \leq U^1_\infty$. By construction we see that, on the event where the first two transitions are a mutation to type $a$ of the non-grey line and a mutation to type $A$ of the grey line, and where these two event occur on the time interval $[0,1]$ (this event has positive probability because of the assumption \eqref{eq:bimut}), we have $H(x,V^2_r)=0$ and $H(x,V^1_r)=1$ for all $r\geq 1$, so in particular $U^2_\infty=0<1=U^1_\infty$. The result follows
\end{proof}

\section{Some facts about L\'evy processes} \label{factslp}

The following is a well-known result about L\'evy processes. It shows that if a L\'evy process satisfies certain integrability conditions, then its supremum has finite exponential moments. 

\begin{lemme} \label{expomomentsup}
Let $X$ be a real-valued L\'evy process with Laplace exponent $\psi_X$. If there is $\lambda_0 > 0$ in the domain of the Laplace transform of $X$ such that $\psi_X(\lambda_0)<0$, then $X$ drifts to $-\infty$ (thus, $\sup_{t \in [0,\infty)} X_t$ is well-defined and finite) and, for any $\lambda \in (0,\lambda_0)$, we have $\mathbb{E}[e^{\lambda \sup_{t \in [0,\infty)} X_t}]<\infty$. 
\end{lemme}

\begin{proof}
For the first part, note that, for any $t \geq 1$, we have $$\mathbb{P}(X_t \geq 0) \leq \mathbb{E}[e^{\lambda_0 X_t}] = e^{t \psi_X(\lambda_0)}.$$ 
Hence, $$\int_1^{\infty} t^{-1}\mathbb{P}(X_t \geq 0) \dd t<\infty,$$which entails that $X$ drifts to $-\infty$ by Rogozin's criterion (see for example \cite[Thm. VI.12]{Bertoin}). Moreover, we know from \cite[Thm. 48.1]{Sato} that $\sup_{t \in [0,\infty)} X_t$ is an infinitely divisible random variable with L\'evy measure $\nu(\dd x)$ given by $\nu(\dd x) \coloneqq \mathds{1}_{x \geq 0} \int_0^{\infty} t^{-1}\mathbb{P}(X_t \in \dd x) \dd t$. Let us fix $\lambda \in (0,\lambda_0)$. Thanks to \cite[Thm. 25.3]{Sato}, the result will follow if we prove that $\int_{(1,\infty)} e^{\lambda x} \nu(\dd x) < \infty$. Using the expression of $\nu$, Holder's inequality, the definition of the Laplace exponent $\psi_X$, \cite[Lem. 30.3]{Sato}, and the fact that $\psi_X(\lambda_0)<0$, we get 
\begin{align*}
\int_{(1,\infty)} e^{\lambda x} \nu(\dd x) & = \int_0^{\infty}\!\! t^{-1}\mathbb{E}[e^{\lambda X_t} \mathds{1}_{\{X_t >1\}}] \dd t \leq \int_0^{\infty} t^{-1}(\mathbb{E}[e^{\lambda_0 X_t}])^{\frac{\lambda}{\lambda_0}} \times (\mathbb{P}(X_t >1))^{\frac{\lambda_0-\lambda}{\lambda_0}} \dd t \\
& \leq \int_0^{\infty} t^{-1}e^{t \lambda \psi_X(\lambda_0)/\lambda_0} (C t)^{\frac{\lambda_0-\lambda}{\lambda_0}} \dd t = C^{\frac{\lambda_0-\lambda}{\lambda_0}} \int_0^{\infty} t^{-\frac{\lambda}{\lambda_0}}e^{t \lambda \psi_X(\lambda_0)/\lambda_0} \dd t < \infty. 
\end{align*}
In the above, $C$ is the constant provided by \cite[Lem. 30.3]{Sato}. This concludes the proof. 
\end{proof}

\begin{lemme} \label{controlmomentslp}
Let $X$ be a real-valued L\'evy process. If there is $\alpha > 1$ such that $\mathbb{E}[|X_1|^{\alpha}]<\infty$, then there is a constant $C>0$ such that, for any $t\geq 0$, we have $$\mathbb{E}[|X_t|^{\alpha}]<C (\mathds{1}_{\{t\in [0,1)\}} + t^{\alpha} \mathds{1}_{\{t\geq 1\}}).$$
\end{lemme}

\begin{proof}
Let us fix $t \geq 0$ and let $n\coloneqq\lfloor t \rfloor$. Using the triangular inequality and Jensen's inequality, we get 
\begin{align*}
|X_t|^{\alpha} & \leq (n+1)^{\alpha} \left ( \frac{|X_1| + |X_2 - X_1| + \ldots + |X_n - X_{n-1}| + |X_t - X_n|}{n+1} \right )^{\alpha} \\
& \leq (n+1)^{\alpha-1} \left ( |X_1|^{\alpha} + |X_2 - X_1|^{\alpha} + \ldots + |X_n - X_{n-1}|^{\alpha} + |X_t - X_n|^{\alpha} \right ). 
\end{align*}
Since the increments of $X$ are independent and stationary, we obtain 
\begin{align*}
\mathbb{E}[|X_t|^{\alpha}] & \leq (n+1)^{\alpha-1} \left ( n \mathbb{E}[|X_1|^{\alpha}] + \mathbb{E}[|X_{t-n}|^{\alpha}] \right ) \\
& \leq (n+1)^{\alpha} \sup_{s \in [0,1]} \mathbb{E}[|X_s|^{\alpha}] \leq (n+1)^{\alpha} \mathbb{E} \bigg[ \sup_{s \in [0,1]} |X_s|^{\alpha} \bigg]. 
\end{align*}
Applying \cite[Thm. 25.18]{Sato} to the function $g(x)\coloneqq(|x| \vee 1)^{\alpha}$, we get that the assumption $\mathbb{E}[|X_1|^{\alpha}|<\infty$ implies the finiteness of $\mathbb{E} [ \sup_{s \in [0,1]} |X_s|^{\alpha} ]$. Then, we have $n+1 \leq 2t$ when $t \geq 1$ and $n+1=1$ when $t \in [0,1)$. The result follows. 
\end{proof}

\begin{lemme} \label{polmomentvaluesup}
Let $X$ be a real-valued L\'evy process such that $\mathbb{E}[X_1]<0$ (in this case $X$ drifts to $-\infty$, and thus, $\sup_{t\in[0,\infty)} X_t$ is well-defined and finite). Let $\Pi_X$ denote the L\'evy measure of $X$. If there is $m > 1$ such that $\int_{\mathbb{R}\setminus[-1,1]}|x|^{m}\Pi_X(\dd x)<\infty$, then for any $r \in (0,m(m-1)/(2m-1))$, we have $\mathbb{E}[(\sup_{[0,\infty)} X)^{r}]<\infty$. 
\end{lemme}

\begin{proof}
By \cite[Thm. 48.1]{Sato}, $\sup_{t\in[0,\infty)} X_t$ is an infinitely divisible random variable with L\'evy measure $$\nu(\dd x) \coloneqq \mathds{1}_{x \geq 0} \int_0^{\infty} t^{-1}\mathbb{P}(X_t \in \dd x) \dd t.$$ 
Let us fix $r \in (0,m(m-1)/(2m-1))$. According to \cite[Thm. 25.3]{Sato}, we only need to prove that $\int_{(1,\infty)} x^r \nu(\dd x) < \infty$. 

Let $\gamma \coloneqq r(2m-1)/m$ and note that $\gamma\in (0,m-1)$. When $t\geq 1$ we apply \cite[Lem. C.2]{cordhumvech2024} and when $t\in [0,1)$ we apply \cite[Lem. 30.3]{Sato}. Altogether we obtain the existence of a constant $C_1>0$ such that 
\begin{align}
\mathbb{P} (X_t \geq 1) \leq C_1 \left ( t \mathds{1}_{\{t\in [0,1)\}} + t^{-\gamma} \mathds{1}_{\{t\geq 1\}} \right ). \label{polmomentsup1}
\end{align}
Now choose $p,q$ such that 
$$p \in (1,m /r),\quad q \in (1,2-m^{-1}),\quad\text{and}\quad 1/p+1/q=1.$$ Such a choice of $p$ and $q$ is possible since $$\frac{r}{m}+ \frac{1}{2-m^{-1}} < \frac{m-1}{2m-1} + \frac{m}{2m-1} =1.$$ By \cite[Thm. 25.3]{Sato}, the assumption $\int_{\mathbb{R}\setminus[-1,1]}|x|^{m}\Pi_X(\dd x)<\infty$ implies $\mathbb{E}[|X_1|^{m}|<\infty$ and, since $pr<m$, we have $\mathbb{E}[|X_1|^{pr}|<\infty$. Using Holder's inequality for our choice of $p$ and $q$, Lemma \ref{controlmomentslp} with $\alpha=pr$ (let $C_2$ be the constant provided by that lemma), and \eqref{polmomentsup1} we get 
\begin{align*}
\mathbb{E} \left [ X_t^r \mathds{1}_{\{X_t \geq 1\}} \right ] \leq \mathbb{E} \left [ |X_t|^{pr} \right ]^{1/p} \mathbb{P} (X_t \geq 1)^{1/q} \leq C_1^{1/q} C_2^{1/p} \left ( t^{1/q} \mathds{1}_{\{t\in [0,1)\}} + t^{r-\gamma/q} \mathds{1}_{\{t\geq 1\}} \right ). 
\end{align*}
Therefore, 
\begin{align*}
\int_1^{\infty}\!\!\!\! x^r \nu(\dd x) = \int_0^{\infty} \frac{\mathbb{E} \left [ X_t^r \mathds{1}_{\{X_t \geq 1\}} \right ] }{t}\dd t \leq C_3  \left ( \int_0^1\!\! t^{-1+1/q} \dd t + \int_1^{\infty}\!\!\! t^{-1+r-\gamma/q} \dd t \right ) < \infty. 
\end{align*}
with $C_3=C_1^{1/q} C_2^{1/p}$. The finiteness comes from the fact that $q < 2-m^{-1} = \gamma/r$ which implies $r-\gamma/q<0$. 
\end{proof}

\section{Classical facts about Radon distance and Wasserstein distance} \label{classicalfacts}

The space of continuous functions $f : [0,1] \rightarrow \mathbb{R}$ is denoted by $\mathcal{C}([0,1])$. The subspace of $\mathcal{C}([0,1])$ consisting of $1$-Lipschitz functions is denoted by $\mathcal{L}_1([0,1])$. We denote by $d_{Rad} (\cdot,\cdot)$ the Radon distance on probability measures on $[0,1]$. More precisely, for two probability measures $\nu_1, \nu_2$ on $[0,1]$, 
\begin{align}
d_{Rad} (\nu_1, \nu_2)\coloneqq \sup_{f \in \mathcal{C}([0,1]), ||f||_{\infty} \leq 1} \left | \int_{[0,1]} f(z) \nu_1(\dd z) - \int_{[0,1]} f(z) \nu_2(\dd z) \right |. \label{defradondist}
\end{align}
We also denote by $W_p (\cdot,\cdot)$ the Wasserstein $p$-distance on probability measures on $[0,1]$. More precisely, for two probability measures $\nu_1, \nu_2$ on $[0,1]$, 
\begin{align}
W_p (\nu_1, \nu_2)\coloneqq \left ( \inf_{\pi \in \Pi(\nu_1,\nu_2)} \int_{[0,1]} |z_1-z_2|^p \pi(\dd z_1,\dd z_2) \right )^{1/p}, \label{defwassdist}
\end{align}
where $\Pi(\nu_1,\nu_2)$ is the set of probability measures on $[0,1]^2$ whose marginals are given by $\nu_1$ and $\nu_2$. We recall the classical inequality 
\begin{align}
d_{LP}(\nu_1,\nu_2)^2 \leq W_p(\nu_1,\nu_2)^p, \label{classicalLP}
\end{align}
where $d_{LP}$ denotes the L\'evy-Prokhorov metric on probability measures on $[0,1]$. 

Since we are considering measures on $[0,1]$ which is bounded, Kantorovich and Rubinstein's theorem provides the following dual representation for $W_1 (\cdot,\cdot)$: 
\begin{align}
W_1 (\nu_1, \nu_2) = \sup_{f \in \mathcal{L}_1([0,1])} \left | \int_{[0,1]} f(z) \nu_1(dz) - \int_{[0,1]} f(z) \nu_2(\dd z) \right |. \label{kantwassdist}
\end{align}
From \eqref{kantwassdist} and \eqref{defradondist} it is not difficult to see that we have the following classical inequality between Radon distance and Wasserstein distance: 
\begin{align}
2 W_1 (\nu_1, \nu_2) \leq d_{Rad} (\nu_1, \nu_2). \label{inegradonwass}
\end{align}
In particular, \eqref{classicalLP} and \eqref{inegradonwass} imply that $W_p$ and $d_{Rad}$ are stronger than the L\'evy-Prokhorov metric $d_{LP}$. 

\end{appendix}

\subsection*{Acknowledgment}
The authors would like to thank Sebastian Hummel for insightful discussions at an early stage of the project.

\subsection*{Funding}
Fernando Cordero was funded by the Deutsche Forschungsgemeinschaft (DFG, German Research Foundation) --- Project-ID 317210226 --- SFB 1283. Gr\'egoire V\'echambre was funded by Beijing Natural Science Foundation, project number IS24067. 

\addtocontents{toc}{\protect\setcounter{tocdepth}{2}}
\bibliographystyle{abbrvnat}
\bibliography{reference2}

\end{document}